\documentclass[reqno,final]{amsart}
\usepackage{natbib}  
\usepackage{fancyhdr} 
\usepackage{color} 
\usepackage{hyperref} 
\usepackage{graphicx} 

\usepackage{pstricks}
\usepackage{amssymb}

\definecolor{aleacolor}{rgb}{0.16,0.59,0.78}

\hypersetup{
breaklinks,
colorlinks=true,
linkcolor=aleacolor,
urlcolor=aleacolor,
citecolor=aleacolor}

\renewcommand{\headheight}{11pt}
\renewcommand{\cite}{\citet}

\theoremstyle{plain}
\newtheorem{theorem}{Theorem}[section]                                          
\newtheorem{proposition}[theorem]{Proposition}                          
\newtheorem{lemma}[theorem]{Lemma}
\newtheorem{corollary}[theorem]{Corollary}

\theoremstyle{definition}
\newtheorem{definition}[theorem]{Definition}
\theoremstyle{remark}
\newtheorem{remark}[theorem]{Remark}

\makeatletter \@addtoreset{equation}{section} \makeatother
\renewcommand\theequation{\thesection.\arabic{equation}}
\renewcommand\thefigure{\thesection.\arabic{figure}}
\renewcommand\thetable{\thesection.\arabic{table}}

\newcommand{\dsp}{\displaystyle}
\newcommand{\bd}{\begin{displaymath}}
\newcommand{\be}{\begin{equation}}
\newcommand{\beq}{\begin{eqnarray}}
\newcommand{\ba}{\begin{array}}
\newcommand{\ed}{\end{displaymath}}
\newcommand{\ee}{\end{equation}}
\newcommand{\eeq}{\end{eqnarray}}
\newcommand{\ea}{\end{array}}
\newcommand{\espace}{\mbox{ }}

\newcommand{\eps}{\varepsilon}
\newcommand{\N}{{\mathbb N}}

\newcommand{\Z}{{\mathbb Z}}
\newcommand{\Q}{{\mathbb Q}}
\newcommand{\R}{{\mathbb R}}

\newtheorem{notation}[theorem]{Notation}

\let\proof\relax 

\newenvironment{proof}[2]{\espace\\{\em Proof of #1 \ref{#2}.}}{\hfill\mbox{$\square$}}
\begin{document}

\title[Shape theorem for an epidemic model 
in $d\ge 3$]{A shape theorem for an epidemic model in dimension $d\ge 3$}

\author{E. D. Andjel}
\author{N. Chabot}
\author{E. Saada}
\address{Aix-Marseille Universit\'e,
 CNRS, Centrale Marseille, I2M, UMR 7373,\newline
 Technop\^ole Ch\^ateau-Gombert,
39 rue Fr\'ed\'eric Joliot-Curie,\newline
13453 Marseille Cedex 13, France.\newline
Visiting IMPA, Rio de Janeiro, Brasil.}
\email{enrique.andjel@univ-amu.fr}
\thanks{E. D. Andjel was partially supported by
PICS no. 5470, by CNRS and by CAPES.}
\address{Lyc\'ee Lalande,\newline
16 rue du Lyc\'ee,\newline
01000 Bourg en Bresse, France.}
\email{nicolas.chabot@cegetel.net}
\address{CNRS, UMR 8145, MAP5,
Universit\'e Paris Descartes, Sorbonne Paris Cit\'e,\newline
45 rue des Saints-P\`eres,\newline
75270 Paris cedex 06, France.}
\email{Ellen.Saada@mi.parisdescartes.fr}
\urladdr{\url{http://www.math-info.univ-paris5.fr/\~esaada/}}
\thanks{E. Saada was partially supported by
PICS no. 5470}
\begin{abstract}
 We prove a shape theorem  for the set of infected individuals
in a spatial epidemic model with 3 states (susceptible-infected-recovered) 
on $\Z^d,d\ge 3$, when there is no extinction of the infection. 
For this, we derive percolation
estimates (using dynamic renormalization techniques) for  
a locally dependent random graph in correspondence with the epidemic model.
\end{abstract}
\maketitle
\section{Introduction}\label{sec:intro}
\cite{MR0496851,MR0480208} has introduced  a stochastic spatial 
 epidemic model on $\Z^d$ called   ``general
 epidemic model'', 
describing the evolution of individuals submitted to infection by contact
contamination of infected neighbors. More precisely, on each site of $\Z^d$ 
there is an individual who can be healthy, infected, or immune. 
At time 0, there is an infected 
individual at the origin, and all other sites are occupied by healthy 
individuals. 
Each  infected individual emits germs according to a Poisson process, 
it stays infected for a random time, then it recovers and becomes 
immune to further infection.
A germ emitted from $x\in\Z^d$ goes to one of the neighbors $y\in\Z^d$ of 
$x$ chosen at random. If the individual at $y$ is healthy
then it becomes
infected and begins to emit germs; if this individual is infected or 
immune, nothing happens. The germ emission processes and the durations of infections of 
different individuals are mutually independent.\\ \\
After Mollison's papers, this epidemic model has given rise to many 
studies, and  other models that 
are variations
of this ``SIR'' (Susceptible-Infected-Recovered) structure
have been introduced. 
A first direction to study such models 
is whether  the 
different states asymptotically survive or not, according 
to the values of the involved parameters (e.g. the infection and 
recovery rates). A second direction 
 is the obtention of a shape 
theorem for the asymptotic behavior of infected individuals, 
when there is no extinction  of the infection (throughout this paper, 
``extinction'' is understood as ``extinction of the infection'').  \\ \\
 \cite{K} 
proved that for $d=1$, extinction is almost sure for
the general epidemic model. 
 \cite{MR0675138}
has studied the threshold behavior of this model in dimension $d\ge 2$.
He proved that the process has a critical infection rate below which extinction is
almost certain, and above which there is survival,  thus  closing
this question.  
His work (as well as the following ones on this model)
is based on the analysis of an oriented percolation model, 
that he calls a ``locally dependent random graph'', in correspondence 
with the epidemic model.
See also the related paper \cite{MR0766826}.\\ \\
 In the general epidemic model on $\Z^2$, when there is no extinction,
\cite{MR0978353} have derived a shape theorem 
for the set of infected  or immune  individuals 
 when the contamination rule is nearest neighbor, 
 and the durations of infection are positive with a positive probability.
A second moment is required for those durations only to localize the infected but not
immune individuals within the shape obtained.  
This result was extended to a finite range 
contamination rule by  \cite{MR1245289}. The proofs in \cite{MR0978353,MR1245289} 
are based on the correspondence with the locally dependent random graph;
 they refer to
\cite{MR0624685}, which deals with first passage percolation (see also \citealp{MR0876084}),
 including the possibility of infinite passage times. 
They rely on  circuits to delimit and control
open paths. This technique  cannot
be used for dimension greater than $2$.\\ \\ 
 There was no investigation of the shape theorem
for the general epidemic model in higher dimensions, 
until \cite[unpublished]{C}  proved it 
for a nearest neighbor contamination rule in dimension $d\ge 3$,
 with the restriction to deterministic 
durations of infection: in that case the oriented percolation model is comparable
to a non-oriented Bernoulli percolation model (as noticed in \citealp{MR0675138}, the
case with constant durations of infection  in the
epidemic model  is the only one where the edges are independent
 in the percolation model).
 Analyzing the  epidemic   model for  
$d\ge 3$  required  heavier techniques than before: \cite{C} used 
results from \cite{MR1404543} and \cite{MR1068308} 
for non-oriented Bernoulli percolation to derive,  for the percolation model,
 exponential estimates
in the subcritical case on the one hand, and estimates using percolation on slabs
on the other hand. To apply those results to the epidemic model required to find an 
alternative,  in the percolation model,
 to the neighborhoods (for points in 
$\Z^2$) delimited by circuits of \cite{MR0978353}. \cite{C} introduced new types
of random neighborhoods characterized by the properties of the  
percolation model  in dimension $d\ge 3$.  \\ \\ 
In the present work, we complete 
the derivation of the shape theorem 
for  the set of  infected  or immune  individuals 
in the  general epidemic model with a nearest neighbor 
contamination rule in dimension $d\ge 3$, by proving it for random
durations of infection,  which are positive with a positive probability. 
There, the comparison with non oriented 
percolation done in \cite{C} is no longer valid,  and we have to deal
with an oriented dependent percolation model, with possibly infinite passage times.  
  Our approach consists in adapting the dynamic renormalization techniques 
 of \cite{MR1068308} without calling on
\cite{MR1404543}. This simplifies the paper, but  we obtain sub-exponential
estimates (which suffice for our purposes), instead of exponential estimates as in the paper \cite{C}.
 With this in hand, it is then possible to catch hold of the skeleton of the latter:
 We  take advantage of the random neighborhoods  introduced there
 (they turn out to be still valid in our setting) 
to derive the shape theorem.  Similarly to \cite{MR0978353}, we require
a moment of order $d$
of the durations of infection only to localize the infected but not
immune individuals within the shape obtained.  \\ \\
Let us mention two recent works, \cite{CT} and \cite{MR2923190},
on shape theorems for (or related to)
first passage (non dependent) percolation on $\Z^d$ with various assumptions on the
passage times, for which the approach in \cite{MR0876084} is extended.  \\ \\
Our paper is organized as follows. 
In Section \ref{sec:model} we define the  general epidemic model, 
the locally dependent random graph, we explicit their link, 
and we state the shape theorem 
(Theorem \ref{th:shape}).  Section \ref{sec:appliquer_GM} is devoted to
the necessary percolation estimates on the locally dependent 
random graph needed for Theorem \ref{th:shape}. We prove the latter
in Section \ref{sec:Nicolas}, thanks to an analysis of the  travel   
times for the epidemic. In Appendix \ref{sec:appendix}, we prove all the results 
of Section \ref{sec:appliquer_GM} requiring  dynamical renormalisation techniques. 
\section{The set-up: 
definitions and results}\label{sec:model}
Let $d\ge 3$. The epidemic model on $\Z^d$ is represented 
by a Markov process $(\eta_t)_{t\ge 0}$
of state space $\Omega=\{0,i,1\}^{\Z^d}$.  
The value $\eta_t(x)\in\{0,i,1\}$ is the state 
of  the individual located at site  $x$ at time $t$: 
state 1 if the individual is healthy (but not immune), 
state $i$ if it is infected, or state 0
if it is immune.  We will shorten this in ``site $x$ is healthy, infected or immune''. 
 We assume that at time 0, the origin $o=(0,\ldots,0)$ 
is the only infected site while all other sites are healthy. 
That is, the initial configuration $\eta_0$ is given by
\begin{equation}\label{eq:IC}
\eta_0(o)=i,\qquad\forall\, z\not= o, \,\eta_0(z)=1.
\end{equation}
We now describe how the epidemic propagates, 
then  we introduce  
a related  locally dependent oriented bond percolation model on $\Z^d$, 
and finally we link the two models. 
We assume that all the processes and random variables we deal with 
are defined on a common probability space,
whose probability is denoted by $P$, and the corresponding expectation by $E$. \par
For $x=(x_1,\ldots,x_d)\in\Z^d, y=(y_1,\ldots,y_d)\in\Z^d$,
 $\| x-y\|_1=\sum_{i=1}^d |x_i-y_i|$ denotes the $l^1$ norm of 
$x-y$,
 and we write $x\sim y$ if $x,y$ are neighbors, that is $\| x-y\|_1=1$. 
 Let $(T_x, e(x,y): x,y \in \Z^d,  x\sim y)$ be independent random variables 
such that \par
\noindent
1) the $T_x$'s are  nonnegative  with a common distribution  
satisfying  $P(T_x=0)<1$;\par
\noindent 
2) the $e(x,y)$'s are exponentially distributed with a parameter $\lambda>0$. \\
 We stress that the only assumption on the $T_x$'s is that their
distribution is not a Dirac mass on 0. They could be infinite,
or without any finite moment.
 We define 
\begin{equation}\label{def:open-closed-bonds}
X(x,y)=
\begin{cases}
1 &\text{if $e(x,y)<T_x$;}\\
0   &\text{otherwise.}
\end{cases}
\end{equation}\\ 
 In the epidemic model, 
for a given infected individual $x$,  $T_x$ denotes the amount of time $x$ stays infected; 
during this time of infection, $x$ emits germs according to a Poisson process 
of parameter $2d\lambda$; when $T_x$ is over, $x$ recovers and its state becomes 0 forever. 
An emitted germ from $x$  at some time $t$ reaches  $y$ (say), one of the $2d$ 
neighbors of $x$,  uniformly. If this neighbor $y$ 
 is in state 1 at time $t^-$, it immediately changes to state $i$ at time $t$, 
 from $t$ begins the duration of infection $T_y$, and $y$ begins 
to emit germs according to the same rule as $x$ did;  if this neighbor $y$ 
is in state 0 
or $i$ at time $t^-$,  nothing happens. \\ \\
 In the percolation model, for $x,y\in\Z^d,  x\sim y$, 
the oriented bond $(x,y)$ is said to be \textit{open with passage time} 
$e(x,y)$  (abbreviated $\lambda$\textit{-open}, or \textit{open} when the parameter is fixed) 
if $X(x,y)=1$ and \textit{closed} 
(with infinite passage time) if $X(x,y)=0$.  As in \cite{MR0675138}, we call this oriented percolation model
 a {\sl locally dependent random graph}. Indeed the fact that {\sl any} of the bonds exiting from site $x$
is open depends on the r.v. $T_x$.\par
For $x,y\in\Z^d$ (not necessarily neighbors),
``$x\to y$'' means that there exists  (at least) 
{\sl an open path} from $x$ to $y$, that is a path of open oriented bonds,
 $\Gamma_{x,y}=(z_0=x,z_1,\ldots,z_n=y)$. \par
If  $x\to y,\,x\not=y$, 
we define the \textit{passage time on} $\Gamma_{x,y}$ to be 
(see \eqref{def:open-closed-bonds})
\begin{equation}\label{eq:passage-time-on-Gamma}
\overline{\tau}(\Gamma_{x,y})=\sum_{j=0}^{n-1}e(z_j,z_{j+1})
\end{equation}
and, if $x=y$, we put $\overline{\tau}(\Gamma_{x,x})=0$.\par
\noindent 
 We then define the \textit{travel time from $x$ to} $y$  to be
\begin{equation}\label{eq:passage-time-from-x-to-y}
\tau(x,y)=
\begin{cases}
\dsp{\inf_{\{\Gamma_{x,y}\}}}\,\overline{\tau}(\Gamma_{x,y})
 &\text{if  $x\not=y,x\to y$,}\\
0 &\text{if  $x=y$,}\\
+\infty &\text{otherwise.}
\end{cases}
\end{equation}
where the infimum is over all possible open paths from $x$
to $y$. \\ \\
 Coming back to the  epidemic model, note that when the initial 
configuration is $\eta_0$ defined in \eqref{eq:IC}, for a given site $z$,
$\tau(o,z)$ is the duration for the infection to propagate from $o$ to $z$, changing
successively the values on all the sites of the involved path $\Gamma_{o,z}$
from 1 to $i$.\\ \\
 To link the two models, we define, for $t\ge 0$,
\begin{eqnarray}\label{eq:immune_at_t}
\Xi_t &= \{x\in\Z^d: x \mbox{ is immune at time } t\} &= \{x\in\Z^d: \eta_t(x)=0\};\\
\label{eq:infected_at_t}
\Upsilon_t &= \{x\in\Z^d: x \mbox{ is infected at time } t\} &= \{x\in\Z^d: \eta_t(x)=i\}.
\end{eqnarray} 
We have, for $z\in\Z^d,\,t\ge 0$,
\begin{equation}\label{eq:xi-zeta-tau}
 z\in\Upsilon_t\cup \Xi_t \quad\hbox{ if and only if }\quad {\tau}(o,z)\le t.
\end{equation}
Indeed, ${\tau}(o,z)\le t$ means that the infection has reached site $z$
before time $t$, so that site $z$ is either still infected or already
immune at time $t$, that is $z\in\Upsilon_t\cup \Xi_t$. Conversely, if
$z$ is infected or immune at time $t$, it means that it has already been infected.\\ \\
In the epidemic model,  we denote by  $C_o^{\rm out}$  the set of sites that will 
ever become infected, that is
\be\label{eq:C_o}
C_o^{\rm out}=\{x\in\Z^d:\exists\, t\ge 0, \eta_t(x)=i|\eta_0(o)=i,\forall\, z\not= o, \eta_0(z)=1\}.
\ee
 Then, by \eqref{eq:xi-zeta-tau}, $C_o^{\rm out}$ 
is the set of sites that can be reached from the origin following 
an open path in the percolation model. See also \cite[(1.2)]{MR0978353}, 
\cite[p. 322]{MR0496851} and \cite[Lemma 3.1]{MR0675138}.  \\ \\ 
 More generally,  in the percolation model,  for each $x\in \Z^d$ we define the 
\textit{ingoing and outgoing clusters to and from} $x$ to be 
\be\label{def:clusters_of_x}
C_x^{\rm in} =\{y\in \Z^d:y\rightarrow x\},\qquad C_x^{\rm out} =\{y\in \Z^d:x\rightarrow y\},
\ee
and the corresponding critical values to be
\be\label{def:critical_of_x}
\lambda_c^{\rm in}= \inf \{\lambda: P(\vert C_x^{\rm in}\vert =+\infty)>0\},\quad \lambda_c^{\rm out}
= \inf \{\lambda: P(\vert C_x^{\rm out}\vert= +\infty)>0\}, 
\ee  
where 
$|A|$ denotes the cardinality of a set $A$. \par
In Section \ref{sec:appliquer_GM}, we will first prove the following proposition
 about these critical values. 
\begin{proposition}\label{cor:lambda-i=lambda-o}
 We have $\lambda_c^{\rm in}=\lambda_c^{\rm out}.$ 
This common value will be denoted by $\lambda_c=\lambda_c(\Z^d)$.
\end{proposition}
 Assuming that $\lambda>\lambda_c$, the most important
part of our work in Section \ref{sec:appliquer_GM} will then be, 
thanks to dynamic renormalization techniques,
to analyze  for the percolation
model percolation on slabs in Theorem \ref{prop:clusters_rentrant-sortant_infinis}, 
and, through a succession of lemmas, to
establish in Proposition \ref{lem:ap2}
subexponential estimates for  the length of the shortest
path between two points $x$ and $y$ given that $x\to y$.
This will imply (see Remark \ref{rk:csq-ap2}) 
uniqueness of the infinite cluster
of sites connected to $+\infty$. 
Proposition \ref{lem:ap2} contains the crucial properties
we will need on the percolation model to derive 
our main result,  the shape theorem, that
we now state.   
\begin{theorem}\label{th:shape}
Assume $\lambda>\lambda_c$,  and the initial configuration of the epidemic model
$(\eta_t)_{t\ge 0}$ to be given by \eqref{eq:IC}. 
Then there exists a convex subset  $D\subset \R^d$  such that, for 
all $\eps>0$ we have,   for $t$ large enough 
\begin{equation}\label{eq:shape}
\Bigl((1-\eps)tD \cap C_o^{\rm out} \Bigr)\subset  \Bigl( \Xi_t \cup \Upsilon_t\Bigr)
\subset\Bigl( (1+\eps)tD \cap C_o^{\rm out} \Bigr)\mbox{ a.s.} 
 \end{equation}
and if $E(T_o^d)<\infty$ we also have
\begin{equation}\label{eq:couronne}
 \Upsilon_t  \subset\Bigl( (1+\eps)tD \setminus (1-\eps)tD \Bigr) \mbox{ a.s. for}\ t \mbox{ large enough.}
\end{equation} 
\end{theorem}
In other words, the epidemic's progression follows linearly the boundary of a convex set.
 Note that a moment assumption on $T_o$ is only required to localize
the infected individuals, and not for the first part of the theorem, for which there is
no assumption on the distribution of $T_o$. It is also remarkable that the fact that
$T_o$ could be either very small or very large with respect to the exponential variables $e(x,y)$
does not play any role.  \\ \\
 We prove  Theorem \ref{th:shape} in Section \ref{sec:Nicolas}. 
For this, we follow some of the fundamental steps of \cite{MR0978353}, but since in 
dimensions three or higher, circuits are not useful as in dimension 2, 
we had to find other methods of proofs. \\  \\
 By \eqref{eq:xi-zeta-tau},  we have to analyze travel times to prove Theorem \ref{th:shape}. 
 On the percolation model, 
we first construct, in Section \ref{subsec:widetildeC},
 for each site $z\in\Z^d$ a random neighborhood ${\mathcal V}(z)$ in such a way that two 
neighborhoods are always connected by open paths (these neighborhoods 
have to be different from those delimited by circuits of \citealp{MR0978353}). For $z,y\in\Z^d$, 
we show that the travel time $\tau(z,y)$  
 is `comparable'
(in a sense precised in Lemma \ref{lem:comparaison-approximation}) to 
the travel time ${\widehat\tau}(z,y)$ to go from ${\mathcal V}(z)$ to ${\mathcal V}(y)$.
 Then we approximate the travel time between sites by a subadditive process, and
we derive (in Theorem \ref{th:radial-limits} and Section \ref{subsec:extending_mu}) 
a radial limit $\mu(x)$
(for all $x$), which is asymptotically the linear growth speed of the epidemic in direction $x$.
In Theorem \ref{th:widehat_t-serie} we control how ${\widehat\tau}(o,\cdot)$ grows. Finally we 
prove in Theorem \ref{th:At-encadre} an asymptotic shape theorem for 
$\widehat\tau(o,\cdot)$, from which we deduce Theorem \ref{th:shape}.  
\section{Percolation estimates}\label{sec:appliquer_GM}
In this section we collect some results concerning the locally dependent random graph, 
 given by the random variables $(X(x,y),x,y\in \Z^d)$
  introduced in  \eqref{def:open-closed-bonds}. 
  Our goal is to derive subexponential estimates
 in Proposition \ref{lem:ap2}.  
\begin{remark}\label{rk:indepdtvectors} 
Although the r.v.'s $(X(x,y),x,y\in \Z^d)$
are not independent, if we denote by $({\rm e}_1,\ldots,{\rm e}_d)$  
the canonical basis of $\Z^d$, then
the  random vectors 
$\{X(x,x+{\rm e}_1),\ldots,X(x,x+{\rm e}_d),X(x,x-{\rm e}_1),\ldots,X(x,x-{\rm e}_d): x\in \Z^d\}$ 
(in which each component depends on $T_x$) are i.i.d., since two different vectors for $z,y\in\Z^d$
depend respectively on $T_z$ and $T_y$ which are independent. This small dependence
forces us to explain why  and how  some results known for independent percolation remain valid in this context.
 \end{remark} 
\begin{remark}\label{rk:FKG}
The function $X(x,y)$ is increasing in the 
independent random variables $T_x$ 
and $-e(x,y)$. It then follows
that the r.v.'s $(X(x,y): x,y \in \Z^d,  y\sim x)$ satisfy 
the following property:\\
\newline  {\rm (FKG)} 
Let $U$ and $V$ be bounded measurable increasing functions
of the random variables $(X(x_j,y_j): x_j,y_j \in \Z^d,
y_j\sim x_j,  j\in \N)$,   then $E(UV)\ge E(U)E(V)$.
\\
\newline For the proof of this property, we refer to
\cite[Lemma (2.1)]{MR0978353} 
 with the help of \cite[Lemma 4.1 and its Corollary]{MR0115221}
 if $U$ and $V$ depend on a finite number of variables
 $X(x_j,y_j)$, and to \cite[Chapter 2]{MR1707339} to take the limit
 for an infinite number of variables.
 \\
\newline We will use this property
 in the proofs of Theorem \ref{prop:clusters_rentrant-sortant_infinis}, 
 Lemma \ref{lem:APp} and  Lemma \ref{lem:k_x-exp_tail} 
 below for $U,V$ indicator functions involving open paths without loops, 
 thus we will speak of increasing events
rather than increasing functions. 
 \end{remark}
For $n\in\N\setminus\{0\}$, let $B(n)=[-n,n]^d$,  let $\partial B(n)$ 
denote the  \textit{inner vertex boundary} of $B(n)$, that is
\be\label{innervertexboundary}
\partial B(n)= \{x\in\Z^d:x\in B(n), x\sim y \mbox{ for some } y\notin B(n)\};
\ee
 and,
for $x\in\R^d$,  $B_x(n)=x+B(n)$. 
 For $A,R\subset\Z^d$, ``$A\rightarrow R$'' means that 
there exists an open path 
$\Gamma_{x,y}$ from some $x\in A$ to some $y\in R$. 
\begin{theorem}\label{th:outgoing-decay}
(i) Suppose $\lambda <{\lambda}_c^{\rm out}$, 
then there exists  $\beta_{\rm out}>0$ such that for all $n>0$,
$P(o\rightarrow \partial B(n))\leq \exp (-\beta_{\rm out} n).$ 

\noindent
(ii) Suppose $\lambda <{\lambda}_c^{\rm in}$, 
then there exists $\beta_{\rm in}>0$ such that for all $n>0$,
$P(\partial B(n) \rightarrow o)\leq \exp (-\beta_{\rm in} n).$ 
\end{theorem}
Theorem \ref{th:outgoing-decay}\textit{(i)} is a special 
case of \cite[Theorem (3.1)]{MR1624925}, whose proof can 
be adapted to obtain Theorem \ref{th:outgoing-decay}\textit{(ii)}.
It is worth noting that in the context of our paper, by 
Remark \ref{rk:FKG},  \cite[Theorem (3.1)]{MR1624925} can 
be proved using the BK inequality instead of the Reimer inequality
 (see \citealp[Theorems (2.12), (2.19)]{MR1707339}). 
Theorem \ref{th:outgoing-decay} 
yields Proposition \ref{cor:lambda-i=lambda-o}:  \\ 
\begin{proof}{Proposition}{cor:lambda-i=lambda-o}
Suppose $\lambda <{\lambda}_c^{\rm in}$. Then by translation invariance and  
Theorem \ref{th:outgoing-decay}\textit{(ii)}
we have that for any $x\in \partial B(n)$, $P(o\rightarrow x)\leq \exp (-\beta_{\rm in} n)$.
Adding over all points of $\partial B(n)$ we get 
$P(o\rightarrow \partial B(n))\leq K'n^{d-1} \exp  (-\beta_{\rm in} n)$ 
for some constant $K'$,  which implies that 
$\lim_{n\to +\infty} P(o\rightarrow \partial B(n))=0$. Therefore $\lambda\leq \lambda_c^{\rm out}$ and 
$\lambda_c^{\rm in}\leq \lambda_c^{\rm out}$. The other inequality is obtained similarly. 
\end{proof}
\mbox{}\\ \\
{}From now on,  we assume $\lambda>\lambda_c(\Z^d)$ and   define the following events: 
 For   $x,y \in\Z^d,A\subset\Z^d$,
\newline \textit{(i)}  The event $\{ x\rightarrow y \ \mbox {within }\ A \}$ 
consists of all points in our probability space for which 
there  exists an open path 
 $\Gamma_{x,y}=(x_0=x,x_1,\ldots,x_n=y)$ from $x$ to $y$ such that
 $x_j\in A$ for all $j\in\{0,\ldots,n-1\}$. 
 Note that the end point $y$ may not belong to $A$.
\newline \textit{(ii)}  The event $\{x\to y \hbox{ outside } A\}$ 
consists of all points in our probability space for which 
 there exists an open path 
$\Gamma_{x,y}=(x_0=x,x_1,\ldots,x_n=y)$ from $x$ to $y$ such that
none of the $x_j$'s ($j\in\{0,\ldots,n\}$) belongs to $A$. 
\begin{definition}\label{def:within-outside}
For $x\in\Z^d,A\subset\Z^d$ let 
\begin{eqnarray*}\label{eq:C_xA}
C_x^{\rm in}(A) &=&\{y\in A:y\rightarrow x\ \hbox{ within }A \}\qquad \mbox{and}\cr
C_x^{\rm out}(A) &=&\{y\in A:x\rightarrow y\ \hbox{ within }A\}.
\end{eqnarray*}
\end{definition}
Note that  by  this definition $C_x^{\rm in}(A)\subset A$ 
and $C_x^{\rm out}(A)\subset A$.\\ \\
 The rest of this section relies heavily on the techniques
of \cite{MR1068308} or \cite[Chapter 7]{MR1707339}. We assume the reader familiar 
with them.  We postpone to Appendix \ref{sec:appendix} 
the proofs of Theorem \ref{prop:clusters_rentrant-sortant_infinis} and Lemma
\ref{lem:connections} below, which require a thoughtful adaptation 
of \cite[Chapter 7]{MR1707339}
for our context of dependent percolation. 
Nonetheless, it is possible
to go directly to Section \ref{sec:Nicolas}, where these techniques
are no longer used, assuming that Proposition \ref{lem:ap2} holds.\par
\medskip
Next theorem is crucial, it states that there is percolation on slabs. 
\begin{theorem}\label{prop:clusters_rentrant-sortant_infinis}
 Assume $\lambda>\lambda_c$.  
For any $k\in\N\setminus\{0\}$, let  $S_k=\{0,1,\dots,k\}\times\Z^{d-1}$ 
denote the slab of thickness $k$ containing $o$. Then 
for $k$ large enough we have  
$$
\inf_{x\in S_k}P(\vert C_x^{\rm in}(S_k)\vert =+\infty)>0, \quad
\mbox{ and }\quad
\inf_{x\in S_k}P(\vert C_x^{\rm out}(S_k)\vert =+\infty)>0.
$$ 
\end{theorem}
We introduce now some notation  about the shortest
path between two points $x$ and $y$ such that $x\to y$. 
\begin{notation}\label{not:extvertbound-Dxy}
 (a) For $A\subset\Z^d$
 we define the 
{\sl exterior vertex boundary} of $A$ as:
\be\label{exteriorvertexboundary}
\Delta_v A= \{x\in\Z^d:x\notin A, x\sim y \mbox{ for some } y\in A\}.
\ee
(b) If $x\rightarrow y$ let $D(x,y)$ be the smallest number 
of bonds required to build an open path from $x$ to $y$ 
 (hence in this path there is no loop, and the $D(x,y)$ bonds are distinct). 
 If $x\not\rightarrow y$, we put $D(x,y)=+\infty$.\par
 \noindent
(c) For $A\subset\Z^d$, $x\in A,y\in \Delta_v A$, 
``$D(x,y)<m \mbox{ within }\ A$''
 means that there is an open path $\Gamma_{x,y}$ 
 using less than $m$ bonds from 
$x$ to $y$ whose sites are all in $A$ except $y$.
\end{notation}
\medskip
 The end of this section provides some upper bounds for the tail of the 
conditional distribution of $D(x,y)$ given the event $\{x\rightarrow y\}$. 
We derive Proposition \ref{lem:ap2}, required in Section \ref{sec:Nicolas},
thanks to Lemmas \ref{lem:connections}, \ref{lem:reaching faces}, \ref{lem:APp}.
These estimates are not optimal
and better results could be obtained by a thoughtful adaptation of the methods of \cite{MR1404543}. 
Instead of getting exponential decays in $\|x-y\|_1$ (or in $n$) we get 
exponential decays in
$\|x-y\|_1^{1/d}$ (or in $n^{1/d}$). We have adopted this approach because 
those weaker results suffice for our purposes and are simpler to obtain,
thus making our proof much easier to follow: 
it is possible to read our work knowing only
\cite{MR1068308} and not \cite{MR1404543}.  Next lemma is inspired by
\cite[Section 5\textit{(f)} p. 454]{MR1068308}. 
\begin{lemma}\label{lem:connections}
 Assume $\lambda>\lambda_c$.  
There exist $\delta>0$,  $k\in \N\setminus\{0\}$ and ${\rm C}_1={\rm C}_1(k)>0$  such that

\noindent
(i) $\forall n>0,\ x\in B(n+k)\setminus B(n),\ y\in (B(n+k)\setminus B(n))
\cup \Delta_v(B(n+k)\setminus B(n) ) $ we have : 
$$P(x\rightarrow y \ \mbox{ within }\ B(n+k)\setminus B(n))>\delta. $$
(ii) Let  for $(n,m)\in\Z^2$ with $n<m$, and for $\ell\ge 0$,
\begin{eqnarray}
A(n,m,\ell)&=&\{z: -k+n\leq z_1<n, -\infty<z_2\leq \ell+k\}\cup\cr
&& \{z:-k+n\leq z_1\leq m+k, \ell<z_2\leq \ell+k\}\cup \cr
&&\{z:m<z_1\leq m+k, -\infty<z_2\leq \ell+k\}.\label{eq:A_nm}
\end{eqnarray}
$\forall n<m,\ \forall x \in A(n,m,0),\forall y\in A(n,m,0)\cup \Delta_v A(n,m,0)$,  we have:
$$P(D(x,y)<{\rm C}_1(\| x-y\|_1 +(-x_2)^+ +(-y_2)^+)\ \mbox{ within } A(n,m,0))>\delta.$$
\end{lemma} 
 We again introduce some notation, to decompose in Lemma \ref{lem:reaching faces}
a path from the center of a box to its boundary through hyperplanes. 
\begin{notation}\label{not:HnJnGn}
Let $k$ be given by Lemma \ref{lem:connections} and let $x$ and $y$ be points in $\Z^d$.
For $\ell\in \Z$ let  $H_\ell=\{z\in \Z^d:z_1=\ell\}$  and define the events, for $n\in\N$,
\begin{eqnarray*}
 J_n&=&\{x\rightarrow H_{x_1-1-jk}\ \mbox{within}\ B_x(nk),  
 j=0,\dots, \lfloor n/2\rfloor\}\cap \cr
&&\quad\{ H_{y_1+1+jk}\rightarrow y\ \mbox{within}\ B_y(nk), 
j=0,\dots, \lfloor n/2\rfloor\},\cr
G_n&=&\{x\rightarrow \partial B_x(nk), \ \partial B_y(nk)\rightarrow y\},
\end{eqnarray*}
where, for any $a\in\R$, $\lfloor a\rfloor$ denotes the greatest integer not greater than
$a$.
\end{notation}
\begin{lemma}\label{lem:reaching faces} 
 Assume $\lambda>\lambda_c$.  
Let $k$ be given by Lemma \ref{lem:connections} and let $x,y$ be points in $\Z^d$.
Then, for $n\in \N\setminus\{0\}$ there exists $\beta>0$ such that
\[
P(J_n\vert G_n)\ge 1-\exp(-\beta n).
\]
\end{lemma} 
\begin{proof}{Lemma}{lem:reaching faces}
 By translation invariance we may assume that $x$ is the origin.
We start showing that for some constant $\beta'>0$ and all $n$
 \begin{eqnarray} 
&&
P(o\rightarrow H_{-1-jk}\ \mbox{within}\ B(nk),  
 j=0,\dots, \lfloor n/2\rfloor  {\vert} o\rightarrow \partial B(nk))\cr
&& \ge 1-\exp(-\beta' n).\label{inter}
\end{eqnarray}
For this we first observe that
\begin{eqnarray*}
&&\{o\rightarrow H_{-1-jk}\ \mbox{within}\ B(nk) 
\mbox{ for some  }\lfloor n/2\rfloor \leq j\leq n \} \\
&&\subset \{  o\rightarrow H_{-1-jk}\ \mbox{within}\ B(nk),  
 j=0,\dots, \lfloor n/2\rfloor  \}.
 \end{eqnarray*}
Hence \eqref{inter} follows from
\begin{eqnarray*}\label{inter'}
&&
P(o\rightarrow H_{-1-jk} \mbox{ within } B(nk) 
\mbox{ for some  }\lfloor n/2\rfloor \leq j\leq n  
{\vert} o\rightarrow \partial  B(nk))\cr
&&\ge 1-\exp(-\beta' n),
\end{eqnarray*}
which is a consequence of Lemma \ref{lem:connections}\textit{(i)}.
Since $P(\partial B_y(kn)\rightarrow y)$ is 
bounded below as $n$ goes to infinity, \eqref{inter} implies that
$$P(o\rightarrow H_{-1-jk}\ \mbox{ within}\ B(nk),  
 j=0,\dots, \lfloor n/2\rfloor  {\vert} G_n )$$
converges to $1$ exponentially fast. Similarly one proves that
$$  P(H_{y_1+1+jk}\rightarrow y\ \mbox{within}\ B_y(nk), 
j=0,\dots, \lfloor n/2\rfloor \vert G_n) $$
converges to $1$ exponentially fast, and the lemma follows.
\end{proof}
\mbox{}\\ \\
 In
Lemma \ref{lem:APp} below we prove a chemical distance bound 
that will be used later on to derive in Remark \ref{rk:csq-ap2}, through 
Proposition \ref{lem:ap2}, the uniqueness of the infinite cluster
of sites connected to $+\infty$. The main
technique is to construct an open path in a ring after independent
attempts thanks on the one hand to 
Lemma \ref{lem:reaching faces} whose $J_n$'s enable to
get disjoint slabs, and on the other hand
to Lemma \ref{lem:connections}\textit{(ii)} 
once we find the appropriate ring. 
\begin{lemma}\label{lem:APp}
 Assume $\lambda>\lambda_c$.  
Let $k$ be given by Lemma \ref{lem:connections}, and let $G_n$ be as in Lemma \ref{lem:reaching faces}.
Then, there exist constants 
${\rm C}_2$, ${\rm C}_3$ and $\alpha_2>0$ such that, for all $x,y \in \Z^d$, 
 $n\in \N\setminus\{0\}$,  we have 
$$P(D(x,y)> {\rm C}_2 \|x-y\|_1+{\rm C}_3(nk)^d \vert \ G_n)\leq \exp(-\alpha_2 n).  $$
\end{lemma}
\begin{proof}{Lemma}{lem:APp}
 Again, by translation invariance we may assume that $x$ is the origin 
 and without loss of generality, we also assume that $y_1>0$ and $y_2\ge 0$.  
By Lemma \ref{lem:reaching faces} it suffices to show that 
$$P(D(o,y)> {\rm C}_2 \|y\|_1+{\rm C}_3(nk)^d \vert \ J_n)$$
 decays exponentially in $n$.

 For $0\le j\le \lfloor n/2\rfloor$, let (see \eqref{eq:A_nm})
$A_j=A(-jk,y_1+jk,y_2+jk)$. 
\begin{figure}[htp]\label{fig:dessin_lemme3.4-2}
\centering
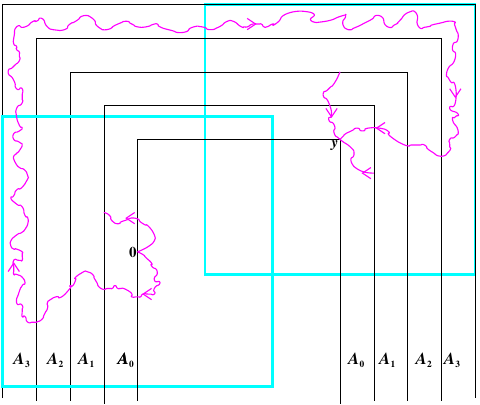
\caption{the event $W_3$}
\end{figure}
Note that the sets  $A_0,\dots, A_{\lfloor n/2\rfloor}$  are disjoint. Figure 1
 should help the reader to visualize them.  Our aim is to 
 find paths from $o$ to $y$ through independent attempts, which will enable 
to use Lemma \ref{lem:connections}\textit{(ii)} in each set $A_j$. 
This is why  
we have first replaced $G_n$ by $J_n$ to condition with.\par  
On the event $J_n$,  we can reach from the origin 
each of the sets $A_i$ by means of an open path contained in $B(nk)$
and from each of these sets we can reach $y$ by means of an open 
path contained in  $ B_y(nk)$.  Hence, on $J_{n}$ for each 
$j\in \{0,\dots,\lfloor n/2\rfloor\}$  there exist a random point 
$U_j\in B(nk)\cap A_j$ and  an open path from $o$ to 
$U_j$ such that all its sites except $U_j$ are in 
$B(nk)\cap (\cap_{\ell=j}^{\lfloor n/2\rfloor}A_\ell^c)$. 
If there are many possible values of $U_j$  we choose 
the first one in some arbitrary deterministic order. Similarly, 
there is a random point  $V_j\in B_y(nk)\cap \Delta_v A_j$ 
and an open path from $V_j$ to $y$ with all its
sites in  $B_y(nk)\cap (\cap_{\ell=j}^{\lfloor n/2\rfloor}A_\ell^c)$. 
Let $u^j$ and $v^j$ be possible values of $U_j$ and $V_j$ respectively. 
Then let $\rm C_1$ be as in Lemma \ref{lem:connections} and define
\begin{eqnarray*}
F_j(u^j,v^j)&=&\{U_j=u^j,V_j=v^j\}, \cr
E_j(u^j,v^j)&=&\{D(u^j,v^j)< {\rm C_1}(\| u^j-v^j\|_1+\vert u^j_2\vert 
+\vert v^j_2\vert)\ \mbox{within}\ A_j\}\   \mbox{and} \cr
W_j&=&\cup_{u^j,v^j}\left(F_j(u^j,v^j)\cap E_j(u^j,v^j)\right),
\end{eqnarray*}
where the union is over all possible values of $U_j$ and $V_j$.
Now we define a subset of $\Z^d$ 
\begin{equation}
 R_j=\Big(B(nk)\cup B_y(nk)
 \cup( A_0\cup\dots \cup A_{j-1})\Big)\cap \Big( A_j^c\cap \dots \cap A_{n-1}^c\Big),
\end{equation}
 and we denote by $\sigma_j$ the $\sigma$-algebra generated by  $\{T_x,e(x,y):x \in R_j,x\sim y\}$.
Then, noting that $\bold{1}_{F_j(u^j,v^j)}\Pi_{\ell=0}^{j-1}\bold{1}_{W_\ell^c} $ is $\sigma_j$-measurable, 
write for  $j=1,\dots, \lfloor n/2\rfloor$: 
\begin{eqnarray}
&& P\left(W_j\cap J_n \cap(
\cap_{\ell=0}^{j-1}W_\ell^c)\right)=\sum_{u^j,v^j}E\left(\bold{1}_{F_j(u^j,v^j)}\mathbf{1}_{E_j(u^j,v^j)}
\bold{1}_{J_n}(\Pi_{\ell=0}^{j-1}\bold{1}_{W_\ell^c})\right)\nonumber\\
&=&
\sum_{u^j,v^j}E\left(\bold{1}_{F_j(u^j,v^j)}(\Pi_{\ell=0}^{j-1}\bold{1}_{W_\ell^c}) 
E( \bold{1}_{J_n} \bold{1}_{E_j(u^j,v^j)}\vert \sigma_j)\right)\nonumber\\
&\ge&
\sum_{u^j,v^j}P(E_j(u^j,v^j))E \left( \bold{1}_{F_j(u^j,v^j)}(\Pi_{\ell=0}^{j-1}\bold{1}_{W_\ell^c}) 
E( \bold{1}_{J_n}  \vert \sigma_j)\right)\nonumber\end{eqnarray} 
\begin{eqnarray}&=&\sum_{u^j,v^j}P(E_j(u^j,v^j)) 
E\left( \bold{1}_{F_j(u^j,v^j)}(\Pi_{\ell=0}^{j-1}\bold{1}_{W_\ell^c}) \bold{1}_{J_n} \right)\nonumber\\
&\ge&
\delta \sum_{u^j,v^j}P\left(F_j(u^j,v^j)\cap J_n \cap(\cap_{\ell=0}^{j-1}W_\ell^c)\right) 
=\delta P\left(J_n\cap(\cap_{\ell=0}^{j-1}W_\ell^c)\right), \label{last-en-plus}
\end{eqnarray}
where the sums are over all possible values of $U_j$ and $V_j$, 
the first inequality follows  from Remark \ref{rk:FKG}
since both 
$J_n$ and  $E_j(u^j,v^j)$ are increasing events, and from the fact that $E_j(u^j,v^j)$ 
is independent of $\sigma_j$;   
 the second inequality follows  from
Lemma 
\ref{lem:connections}\textit{(ii)} and the last equality from the fact that $J_n$ 
is contained in the union of the $F_j(u^j,v^j)$'s which are disjoint. 
We rewrite \eqref{last-en-plus} as
\begin{eqnarray*}
P\left(J_n\cap(\cap _{\ell=0}^{j}W_\ell^c)\right)&\le& 
(1-\delta)P\left(W_j\cap J_n\cap(\cap _{\ell=0}^{j-1}W_\ell^c)\right)\cr
&\le& 
(1-\delta)P\left(J_n\cap(\cap _{\ell=0}^{j-1}W_\ell^c)\right)
\end{eqnarray*}
Now, proceeding by
induction  on $j$ one gets: 
$$P\left(J_n\cap(\cap _{\ell=0}^{\lfloor n/2\rfloor-1}W_\ell^c)\right)
\le (1-\delta)^{\lfloor n/2\rfloor}P(J_n).$$
 Since we can choose ${\rm C}_2$ and ${\rm C}_3$ in such a way that the event 
$\{D(o,y)> {\rm C}_2 \|y\|_1+{\rm C}_3(nk)^d \}$  does not occur if any of 
the $W_i$'s occurs, the lemma follows.
\end{proof}
\mbox{}\\ \\
Next proposition concludes this section.
\begin{proposition}\label{lem:ap2} 
 Assume $\lambda>\lambda_c$.  

(i) Let  
${\rm C}_2$ be as in Lemma \ref{lem:APp}. Then, there exists $\alpha_3>0$ such that
 for all $x,y\in\Z^d,n\in\N$, we have 
$$P(D(x,y)\geq {\rm C}_2\| x-y\|_1+n^d\vert x\rightarrow y)\leq \exp(-\alpha_3 n);$$
(ii) $P(x\rightarrow y \vert \ \vert C_x^{\rm out} \vert =+\infty, \ \vert C_y^{\rm in} \vert =+\infty )=1$.
\end{proposition}  
 \begin{proof}{Proposition}{lem:ap2} 
\textit{(i)} 
Modifying the constant $\alpha_2$, 
the statement of Lemma \ref{lem:APp} above  holds for ${\rm C}_3=1/k^d$. \par
\textit{(ii)}
We have that $\{x\to \infty\hbox{ and } \infty \to y \}=\cap_n G_n$. Hence for all $k$,
\begin{eqnarray*}
P(D(x,y)=+\infty, x\to \infty\hbox{ and } \infty \to y )  
&\leq&
P(D(x,y)=+\infty, G_k) \\
&\leq& P(D(x,y)=+\infty \vert G_k),
\end {eqnarray*}
which converges to 0 when $k$ goes to infinity by Lemma \ref{lem:APp}.
We thus have $P(D(x,y)=+\infty \vert x\to \infty\hbox{ and } \infty \to y )=0$.
\end{proof}
\section{ The shape theorem }\label{sec:Nicolas}
 In the percolation model, 
let $C^\infty$  be the cluster of sites connected to $\infty$:
\begin{equation}\label{eq:widetildeC}
C^\infty= \{x\in\Z^d: x\to \infty \hbox{ and }
\infty\to x\}.
\end{equation}
\begin{remark}\label{rk:csq-ap2}
As a consequence of  Proposition \ref{lem:ap2}\textit{(ii)}, 
$C^\infty$ is a connected set:  if two sites $x,y$ of
$\Z^d$ belong to $C^\infty$, then $x\to y$ and $y\to x$.
\end{remark}  
\subsection{Neighborhoods in $C^\infty$}\label{subsec:widetildeC}
 In this subsection, we construct neighborhoods ${\mathcal V}(\cdot)$ of sites
in $\Z^d$.\par
 We first deal separately with finite clusters, which will have no influence on the
asymptotic shape of the epidemic. We will include them in the neighborhoods
${\mathcal V}(\cdot)$ of sites we construct. 
\begin{definition}\label{def:racines}
For $x\in\Z^d$, let
\[
\begin{cases}
R_x^{\rm out}=\{y\in\Z^d: x\to y \hbox{ outside } C^\infty \}
 &\text{(outgoing root from $x$);}\\
R_x^{\rm in}=\{y\in\Z^d: y\to x \hbox{ outside } C^\infty \}
 &\text{(incoming root to $x$).}
\end{cases}
\]
\end{definition}
 In particular
$x$ belongs to $R_x^{\rm out}$ and $R_x^{\rm in}$ if and only if $x\notin{C^\infty}$.  
Otherwise $R_x^{\rm out}$ and $R_x^{\rm in}$ are empty.
By next lemma, the distribution of the radius of $R_o^{\rm out}\cup R_o^{\rm in}$  decreases exponentially. 
\begin{lemma}\label{lem:exp_decay_R}
There exists $\sigma_1=\sigma_1(\lambda,d)>0$ such that, for all $n\in\N$,
$$P\left((R_o^{\rm out}\cup R_o^{\rm in})\cap \partial B(n)\neq\emptyset\right)\le \exp(- \sigma_1 n).$$
\end{lemma}
\begin{proof}{Lemma}{lem:exp_decay_R}
For $n\in\N\setminus\{0\}$, $R_o^{\rm out}\cap \partial B(2n)\neq\emptyset$ means that 
there exists an open path  $o\to \partial B(2n)$ outside 
$C^\infty$. This implies that there exists $x\in \partial B(n)$ 
satisfying $o\to x\to \partial B(2n)$ outside $C^\infty$. Similarly, 
$R_o^{\rm in}\cap \partial B(2n)\neq\emptyset$ implies that there exists 
$x\in \partial B(n)$ satisfying $\partial B(2n)\to x\to o$ outside 
$C^\infty$. Then for such a point, 
either the cluster $C_x^{\rm out}$ or the cluster $C_x^{\rm in}$ is 
finite, and has a radius larger than or equal to $n$. 
 Relying on Proposition \ref{p1},\textit{b)} in Appendix \ref{sec:appendix}, we can
follow  the proof of \cite[Theorems (8.18), (8.21)]{MR1707339} 
to get the existence 
of $\sigma_0=\sigma_0(\lambda,d)>0$ such that:
\begin{equation}\label{eq:analogue_G-thm8.21}
\begin{cases}
P(C_x^{\rm out} \cap \partial B_x(n)\neq\emptyset, |C_x^{\rm out}| <+\infty) \le  \exp(- \sigma_0 n);\\
P(C_x^{\rm in} \cap \partial B_x(n)\neq\emptyset, |C_x^{\rm in}| <+\infty) \le  \exp(- \sigma_0 n).
\end{cases}
\end{equation}
Hence
\begin{eqnarray*}
P\left((R_o^{\rm out}\cup R_o^{\rm in})\cap \partial B(2n)\neq\emptyset\right) 
&\le& P\left(R_o^{\rm out}\cap \partial B(2n)\neq\emptyset\right)\cr
&&
+P\left(R_o^{\rm in}\cap \partial B(2n)\neq\emptyset\right)\cr
&\le& 2\sum_{x\in\partial B(n)} P(|C_x^{\rm out}| <+\infty, x\to \partial B_x(n))\cr
&&
+2\sum_{x\in\partial B(n)} P(|C_x^{\rm in}| <+\infty,  \partial B_x(n)\to x)\cr
& \le & 4|\partial B(n)| \exp(- \sigma_0 n)
\end{eqnarray*}
which induces the result.
\end{proof}
\mbox{}\\ \\
 To define the neighborhood ${\mathcal V}(x)$ on $C^\infty$ of a site $x$,
we introduce the smallest box whose interior contains $R_x^{\rm out}$ and $R_x^{\rm in}$, 
which contains elements of $C^\infty$, and is such that two elements 
of $C^\infty$ in this box are connected by an open path which does not
exit from a little larger box. For this last condition, which will enable to bound the 
 travel  time through ${\mathcal V}(x)$, 
we use the parameter ${\rm C}_2$ obtained in Lemma \ref{lem:APp}.
\begin{definition}\label{def:k_x}
Let  ${\rm C}'={\rm C_2}d+2$.
Let $\kappa(x)$ be the smallest $l\in\N\setminus\{0\}$
such that 
\[
\begin{cases}
(i)\,\,\,\,\, \partial B_x(l) \cap \left(R_x^{\rm out}\cup R_x^{\rm in}\right)=\emptyset;\\
(ii)\,\,\,  B_x(l) \cap C^\infty \not= \emptyset;\\
(iii)\, \forall\, (y,z) \in (B_x(l) \cap C^\infty)^2,\,y\to z \hbox{ within } B_x({\rm C}'l).
\end{cases}
\] 
\end{definition}
\begin{remark}\label{rk:(i)}
By \textit{(i)} above, $R_x^{\rm out}\cup R_x^{\rm in}\subset B_x(\kappa(x))$.
\end{remark}
 In the next lemma, we bound the probability a box of size $n$ does not 
admit properties \textit{(i)--(iii)} above, that is, we prove that 
the random variable $\kappa(x)$  has  a sub-exponential tail. 
\begin{lemma}\label{lem:k_x-exp_tail}
There exists a constant $\sigma=\sigma(\lambda,d)>0$ such that, for any $n\in\N$,
\[
P(\kappa(x)\ge n)\le \exp(-\sigma n^{1/d}).
\]
\end{lemma}
\begin{proof}{Lemma}{lem:k_x-exp_tail}
We show that the probability that any of the 3 conditions
in Definition \ref{def:k_x} is not achieved for $n$  decreases
exponentially in $n^{1/d}$: \par
\noindent
\textit{(i)} By  translation invariance,  we have by  Lemma \ref{lem:exp_decay_R},
\be\label{eq:non_i}
P\left( \partial B_x(n)\cap\left(R_x^{\rm out}\cup R_x^{\rm in}\right)
\not=\emptyset\right) \le  \exp(-\sigma_1 n).
\ee
\noindent
\textit{(ii)}
There exist $k\in\N$,  
  $\sigma_2=\sigma_2(\lambda,d)>0$ such that
for any $n\in\N$, 
\be\label{eq:non_ii}
P(B_x(n) \cap C^\infty=\emptyset) \le  \exp(-\sigma_2\lfloor n/(k+1)\rfloor).
\ee
 Indeed, let $k=k(\lambda,d)$ be large enough for the conclusions 
 of Theorem \ref{prop:clusters_rentrant-sortant_infinis} to hold on the slab $S_k$.
Then we have
\begin{eqnarray*}
& P(B_x(n) \cap C^\infty=\emptyset)
\le P(\forall\,z\in\{ x+ j{\rm e}_1,0\le j\le n\},z\notin C^\infty)\cr
&= P(\forall\,z\in\{ x+ j{\rm e}_1,0\le j\le n\},C_z^{\rm in} 
\hbox{ or } C_z^{\rm out} \hbox{ is finite})
\end{eqnarray*}
We denote by $S_k(l)=\{l(k+1),\cdots,(l+1)(k+1)-1\}\times\Z^{d-1}$ 
for $l\ge 0$ the slab of thickness $k$ to which $z$ belongs. 
If $C_z^{\rm in}$ (or $C_z^{\rm out}$) is finite, so is $C_z^{\rm in}(S_k(l))$ 
(or $C_z^{\rm out}(S_k(l))$). 
Because
$\{ \vert C_z^{\rm in}(S_k(l))\vert =+\infty \}$
and $\{ \vert C_z^{\rm out}(S_k(l))\vert =+\infty \}$ are increasing events
  it follows from  Theorem \ref{prop:clusters_rentrant-sortant_infinis} and
 the FKG inequality   (see 
 Remark \ref{rk:FKG}) that
\begin{eqnarray} 
&\inf_{u\in S_k(l)}P(\vert C_u^{\rm in}(S_k(l))\vert =\vert C_u^{\rm out}(S_k(l))\vert =+\infty)\cr
&\ge \inf_{u\in S_k(l)}\left(P(\vert C_u^{\rm in}(S_k(l))\vert=+\infty)
P(\vert C_u^{\rm out}(S_k(l))\vert =+\infty)\right)>0.\label{eq:bis-infinf}
\end{eqnarray} 
Since events occurring in two different slabs are independent, we have
\begin{eqnarray*}
&& P(\forall\,z\in\{ x+ j{\rm e}_1,0\le j\le n\},z\notin C^\infty)\cr
 &\le & P(\forall\,l\ge 0,\forall\,z\in\{ x+ j{\rm e}_1,0\le j\le n\}\cap S_k(l), \cr
 &&\qquad C_z^{\rm in}(S_k(l)) \hbox{ or } C_z^{\rm out}(S_k(l)) \hbox{ is finite})\cr
  &\le & \left(P(\forall\,z\in\{ j{\rm e}_1,0\le j\le k\},\right.\cr
 &&\left.\qquad C_z^{\rm in}(S_k(0)) \hbox{ or } C_z^{\rm out}(S_k(0)) 
 \hbox{ is finite}\right)^{\lfloor n/(k+1)\rfloor}\cr
  &\le & \exp(-\sigma_2 \lfloor n/(k+1)\rfloor)
 \end{eqnarray*}
with $\sigma_2=\sigma_2(\lambda,d)>0$, independent of $n$, because, 
for $z_0=\lfloor k/2\rfloor{\rm e}_1$, using \eqref{eq:bis-infinf} we have
\begin{eqnarray*}
& P(\exists\,z\in\{ x+ j{\rm e}_1,0\le j\le k\},
|C_z^{\rm in}(S_k(0))|=|C_z^{\rm out}(S_k(0))|=+\infty)\cr
\ge & P(|C_{z_0}^{\rm in}(S_k(0))| = | C_{z_0}^{\rm out}(S_k(0))|=+\infty)
>0.
\end{eqnarray*}  
\noindent
\textit{(iii)} There exists $ \sigma_3=\sigma_3(\lambda,d)>0$ such that
\begin{eqnarray}\label{eq:sigma_3}
& 
P\left(\exists\,
 (y,z) \in (B_x(n) \cap C^\infty)^2,\,y\not\to z \hbox{ within }(B_x({\rm C}'n)\right)\cr
 & \le \exp (- \sigma_3 n^{1/d}).
\end{eqnarray}
Indeed, if no open path from $y$ to $z$ (both in $B_x(n)\cap C^\infty$) is contained in 
$B_x({\rm C}'n)$, then $D(y,z)\ge 2({\rm C}'-1)n$. Given our choice of ${\rm C}'$ 
this implies
that $D(y,z)\ge {\rm C_2}\|y-z\|_1  +n$.  Therefore \eqref{eq:sigma_3} 
follows from 
Proposition \ref{lem:ap2}\textit{(i)}.
\end{proof}
\mbox{} \\ \\
 We define the  (site) neighborhood in $C^\infty$ of $x$  by
\begin{equation}\label{def:calV-x}
{\mathcal V}(x)=B_x(\kappa(x))\cap C^\infty.
\end{equation} 
\begin{remark}\label{rk:calV-a-2-points} 
(a) By Definition \ref{def:k_x}\textit{(ii)}, ${\mathcal V}(x)\not=\emptyset$. \par 
\noindent
(b) By  Definition \ref{def:k_x}\textit{(iii)},
for all $y,z$ in ${\mathcal V}(x)$, there exists 
at least one  open path from 
$y$ to $z$, denoted by $\Gamma^*_{y,z}$, contained in 
 $B_x({\rm C}'\kappa(x))$. If there are several such paths we 
 choose the first one according to some deterministic 
order.
\end{remark}  
We finally 
define an ``edge'' neighborhood $\overline\Gamma(x)$ of $x$: 
\begin{eqnarray}\label{def:barGamma-x}
\overline\Gamma(x)&=&\{(y',z')\subset B_x(\kappa(x)), 
(y',z')\hbox{ open}\}\cup\cr
&&\qquad\{(y',z')\in\Gamma^*_{y,z},y,z\in {\mathcal V}(x)\}.
\end{eqnarray}
 Those neighborhoods satisfy
\begin{equation}\label{eq:borner_calV-x_et_Gamma-x}
{\mathcal V}(x)\subset B_x(\kappa(x));\qquad\overline\Gamma(x)\subset B_x({\rm C}'\kappa(x)).
\end{equation} 
\subsection{ Travel times and  radial limits}\label{subsec:Radial-limits}
 We now come back to the spatial epidemic model. 
 In this subsection, we estimate  the time needed by the epidemic 
 to cover $C^\infty$, taking advantage of the analysis of paths 
 in the percolation model done in Section \ref{sec:appliquer_GM}. 
We first define an approximation for the passage time of the epidemic, then we 
prove the existence of radial limits for this approximation and for the epidemic.
 We will follow for this the spirit of the construction in \cite{MR0978353}.  \\ \\
 By analogy with \cite{MR0624685,MR0978353}  (although neighborhoods in our context
are defined differently),  we   define, for $x,y\in\Z^d$, 
the travel time from ${\mathcal V}(x)$ to ${\mathcal V}(y)$  and the time spent around $x$
to be  (remember \eqref{eq:passage-time-from-x-to-y})
\begin{eqnarray}\label{def:approx-passage-time}
  \widehat{\tau}(x,y)&=&\dsp{\inf_{x'\in{\mathcal V}(x), y'\in{\mathcal V}(y)} \tau(x',y')}
  ;\\\label{def:2-approx-passage-time}
    u(x)&=& 
    \begin{cases}
\dsp{\sum_{(y',z')\in\overline\Gamma(x)} \tau(y',z')}
 &\text{if  $\overline\Gamma(x)\not=\emptyset$,}\\
0 &\text{otherwise.}
\end{cases}
\end{eqnarray}   
 By Remarks \ref{rk:csq-ap2}, \ref{rk:calV-a-2-points}\textit{(a)}, 
$\widehat{\tau}(x,y)$ 
is  finite. 
If ${\mathcal V}(x)\cap{\mathcal V}(y)\not=\emptyset$, then $\widehat{\tau}(x,y)=0$.\\ \\
We now show that if $y\in C_x^{\rm out}\setminus R_x^{\rm out}$,  $\widehat{\tau}(x,y)$ approximates 
$\tau(x,y)$. 
\begin{lemma}\label{lem:comparaison-approximation}
For $x\in\Z^d$, if $y\in C_x^{\rm out}\setminus R_x^{\rm out}$, we have  
\begin{equation}\label{eq:comparaison-approximation}
\widehat{\tau}(x,y) \le \tau(x,y)\le u(x)+\widehat{\tau}(x,y)+u(y).
\end{equation}
\end{lemma}
\begin{proof}{Lemma}{lem:comparaison-approximation}
 Let $\Gamma_{x,y}$ be an open  path 
 from $x$ to $y$ such that $\tau(x,y)=\overline{\tau}(\Gamma_{x,y})$. 
 Since $y\notin R_x^{\rm out}$ this path must intersect $C^\infty$. 
 Let $c_1$ and $c_2$ be the first and last points we 
 encounter in  $C^\infty$ when moving from $x$ to $y$ 
along $\Gamma_{x,y}$.  By  Definition \ref{def:k_x}\textit{(i)}, 
$c_1 \in{\mathcal V}(x)$ and $c_2 \in{\mathcal V}(y)$: indeed (for instance for $c_1$), 
either $x\in{C^\infty}$ and $c_1=x$, or the point $a\in\partial
B_x(\kappa(x))\cap\Gamma_{x,y}$ does not belong to $R_x^{\rm out}$ and $c_1$ 
is the first point on $\Gamma_{x,y}$ between
$x$ and $a$; we might have $c_1=c_2$, if 
${\mathcal V}(x)\cap{\mathcal V}(y)\not=\emptyset$.
We have, denoting by $\vee$  the concatenation of paths,
\[
\Gamma_{x,y}=\Gamma_{x,c_1}\vee\Gamma_{c_1,c_2}\vee\Gamma_{c_2,y}
\]
 where 
 $\Gamma_{x,c_1}$ (resp. $\Gamma_{c_2,y}$) is an open path from 
$x$ to $c_1$ contained in $B_x(\kappa(x))$ (resp. from $c_2$ to $y$ contained in $B_y(\kappa(y))$)
and $\Gamma_{c_1,c_2}$ is an open path from $c_1$ to $c_2$.
We then obtain the first inequality of \eqref{eq:comparaison-approximation} since: 
$$\widehat{\tau} (x,y)\le \overline{\tau}(\Gamma_{c_1,c_2})\le \overline{\tau}(\Gamma_{x,y})=\tau(x,y).$$
To prove the second inequality of \eqref{eq:comparaison-approximation}, 
let $\Gamma_{d_1,d_2}$ be an open path from 
$d_1 \in{\mathcal V}(x)$ to $d_2\in{\mathcal V}(y)$ 
 such that $\overline{\tau}(\Gamma_{d_1,d_2})=\widehat{\tau}(x,y)$.
Since the  open paths $\Gamma_{x,c_1}$ from $x$ to $c_1$ and $\Gamma^*_{c_1,d_1}$ 
(which exists by Remark \ref{rk:calV-a-2-points}\textit{(b)}) from $c_1$ to $d_1$ have edges in 
$\overline{\Gamma}(x)$ (see \eqref{def:barGamma-x}), the  
open path $\Gamma_{x,d_1}=\Gamma_{x,c_1}\vee\Gamma^*_{c_1,d_1}$  from $x$ to $d_1$ satisfies
$\overline{\tau}(\Gamma_{x,d_1})\le u(x)$. Similarly, there is an open path $\Gamma_{d_2,y}$ 
from $d_2$ to $y$ such that $\overline{\tau}(\Gamma_{d_2,y})\le u(y)$.
We conclude with 
$$\tau(x,y)\le \overline{\tau}(\Gamma_{x,d_1})+\overline{\tau}(\Gamma_{d_1,d_2})
+\overline{\tau}(\Gamma_{d_2,y})\le u(x)+\widehat{\tau}(x,y)+u(y).$$
\end{proof}
\mbox{}\\ \\
 We now prove that $\widehat{\tau}(.,.)$ is almost 
subadditive, which will enable us later on in 
Theorem \ref{th:radial-limits} to appeal 
to Kingman's Theorem. 
\begin{lemma}\label{lem:sous-additif} 
For all $x,y,z\in\Z^d$, we have  the \rm{subadditivity property}
\begin{equation}\label{eq:sous-additif}
\widehat{\tau}(x,z)\le\widehat{\tau}(x,y)+u(y)+\widehat{\tau}(y,z).
\end{equation}
\end{lemma}
\begin{proof}{Lemma}{lem:sous-additif}
Let $\Gamma_{a,b}$ be an open path from $a\in{\mathcal V}(x)$ to 
$b\in{\mathcal V}(y)$  such that 
 $\widehat{\tau}(x,y)=\overline{\tau}(\Gamma_{a,b})$. 
 Similarly, let $\Gamma_{c,d}$ be an open path  from
$c\in {\mathcal V}(y)$ to $d\in{\mathcal V}(z)$  such that 
 $\widehat{\tau}(y,z)=\overline{\tau}(\Gamma_{c,d})$ (we 
 might have $a=b$, $c=d$ or $b=c$). Since both $b$ and $c$ 
 are in ${\mathcal V}(y)$ there exists an open path 
 $\Gamma^*_{b,c}$ from $b$ to $c$ such that 
 $\overline{\tau}(\Gamma^*_{b,c})\le u(y)$ 
 (see Remark \ref{rk:calV-a-2-points}\textit{(b)} and \eqref{def:barGamma-x}). 
 The lemma then follows since the concatenation of these three paths is
 an open path from a point of ${\mathcal V}(x)$ to a point of ${\mathcal V}(z)$ and 
$$\widehat{\tau}(x,z)\le \overline{\tau}(\Gamma_{a,b}) 
+ \overline{\tau}(\Gamma^*_{b,c})+\overline{\tau}(\Gamma_{c,d})\le \widehat{\tau}(x,y)
+u(y)+\widehat{\tau}(y,z).$$
\end{proof}
\mbox{}\\ \\
 We introduce a new notation,  for the length of the shortest path between two neighborhoods.   
For $x,y\in \Z^d$, let 
\be\label{eq:barD}
\overline{D}(x,y)=
\inf_{x'\in {\mathcal V}(x),y'\in {\mathcal V}(y)} D(x',y').
\ee
Note that unlike $D(x,y)$, $\overline{D}(x,y)$
 is always finite.  Next proposition corresponds 
 to Proposition \ref{lem:ap2}\textit{(i)}
 for $\overline{D}(x,y)$ instead of $D(x,y)$. 
 It will be used in 
 Lemma \ref{lem:regularite-approx-tau} which follows. 
\begin{proposition}\label{lem:bar-D}
There exist constants ${\rm C}_4$ and $\alpha_4>0$ such that
$$P(\overline{D}(x,y)\ge {\rm C_4}\|x-y\|_1 +n)\le \exp(-\alpha_4 n^{1/d}),
\qquad \forall\, x,y \in \Z^d,n\in \N.$$
\end{proposition}
\begin{proof}{Proposition}{lem:bar-D}
 Let  ${\rm C_2}$ be as in Lemma \ref{lem:APp} and Proposition \ref{lem:ap2}. 
Then 
\begin{eqnarray*}
&&P(\overline{D}(x,y)\ge {\rm C_2}\|x-y\|_1 +(2d+1){\rm C_2}n)\cr\le&& P(\kappa(x)>n)+P(\kappa(y)>n)\cr
&+&P(\overline{D}(x,y)\ge {\rm C_2}\|x-y\|_1 +(2d+1){\rm C_2}n,\kappa(x)\le n, \kappa(y)\le n)\cr
\le&& P(\kappa(x)>n)+P(\kappa(y)>n)\cr
&+&\sum_{x'\in B_x(n),y'\in B_y(n)}P(D(x',y')\ge {\rm C_2}\|x-y\|_1 +(2d+1){\rm C_2}n, x'\rightarrow y' )\cr
  \le&&  P(\kappa(x)>n)+P(\kappa(y)>n)\cr
&+&\sum_{x'\in B_x(n),y'\in B_y(n)}P(D(x',y')\ge {\rm C_2}\|x'-y'\|_1 +{\rm C_2}n, x'\rightarrow y').
\end{eqnarray*}  
The result follows from Proposition \ref{lem:ap2}
and  Lemma \ref{lem:k_x-exp_tail}. 
\end{proof}
\mbox{}\\ \\
Of course, the random variables $u(x)$ and $\widehat{\tau}(x,y)$ 
are almost surely finite. 
But we will need  later on repeatedly  
a better control of their size, provided by our next lemma.
\begin{lemma}\label{lem:regularite-approx-tau}
For all $x,y\in\Z^d$, $r\in\N\setminus\{0\}$, 
$u(x)$ and $\widehat{\tau}(x,y)$ have a finite $r$-th moment.
\end{lemma}
\begin{proof}{Lemma}{lem:regularite-approx-tau}
By Lemma  \ref{lem:k_x-exp_tail}, 
$u(x)$ is bounded above by a sum of passage times 
$e(y,z)$ with $y$ and $z$ in the box $B_x(Y)$, 
where $Y$ is a random variable whose moments are all finite. 
 By Lemmas  \ref{lem:k_x-exp_tail}  and  \ref{lem:bar-D} 
 the same happens to  $\widehat{\tau}(x,y)$
  (if $x'\in {\mathcal V}(x),y'\in {\mathcal V}(y)$ are the sites that achieve
 $\overline{D}(x,y)$, then $\widehat{\tau}(x,y)\le {\tau}(x',y')$).  
 Therefore it suffices to show that if
 $(X_i,i\in \N)$ is a sequence of i.i.d. random variables and 
 $N$ is a random variable taking values in $\N$, then the 
 moments of  $\sum_{i=1}^N X_i$ 
 are all finite if it is the case for both the $X_i$'s and $N$.
  To prove this write:
 \begin{eqnarray*}
E(\vert \sum_{i=1}^N X_i\vert^r)&= &\sum_{n=1}^{\infty}E(\vert X_1+\dots+X_n\vert^r {\bf 1}_{\{N=n\}})\cr
&\leq&\sum_{n=1}^{\infty}[E(\vert X_1+\dots+X_n\vert^{2r})P(N=n)]^{1/2} \cr
&\leq&\sum_{n=1}^{\infty}[E(\vert X_1\vert +\dots+\vert X_n\vert)^{2r}P(N=n)]^{1/2}\cr
&\leq&\sum_{n=1}^{\infty} [n^{2r}{\rm C}_{2r}P(N=n)]^ {1/2}
 \end{eqnarray*}
where the  second line  
comes from Cauchy-Schwartz' inequality,
the factor $n^{2r}$ counts the number of terms in the  development of 
$(\vert X_1\vert +\dots+\vert X_n\vert)^{2r}$ and the constant ${\rm C}_{2r}$  
depends on the distribution of the  $X_i$'s.
As $N$ has all its  moments finite $P(N=n)$ decreases faster than  $n^{-2r-4}$ 
and the sum is finite.
\end{proof}
\mbox{}\\ \\ 
We now construct a process $(\vartheta_\cdot)$ which  will be  subadditive in every
direction, and  will have  a.s., by Kingman's Theorem, a radial limit denoted by $\mu$. We will then 
check that $\widehat\tau(o,\cdot)$  also has, in every
direction, the same radial limit, and we will extend this conclusion
to  $\tau(o,\cdot)$  on the set $C_o^{\rm out}$ of sites that have ever been infected. 
Hence we first prove 
\begin{theorem}\label{th:radial-limits}
For all $z\in\Z^d$, there exists $\mu(z)\in\R^+$
such that almost surely
\begin{eqnarray}\label{eq:lim-widehat_tau}
&&\lim_{n\to +\infty}\frac {\widehat{\tau}(o,nz)}n =\mu(z) \qquad\hbox{   and  }\\
\label{eq:radial-limit_2}
&&\lim_{n\to +\infty}\left[\frac{\tau (o,nz)}n - \mu(z)\right]{\bf 1}_{\{nz\in C_o^{\rm out}\}}=0.
\end{eqnarray}
\end{theorem}
\begin{proof}{Theorem}{th:radial-limits} 
\textit{(i)} 
For all $z\in\Z^d$, $(m,n)\in\N^2$, let
\begin{equation}\label{eq:vartheta_theta}
\vartheta_z(m,n)=\widehat{\tau}(mz,nz)+u(nz).
\end{equation}
The process $(\vartheta_z(m,n))_{(m,n)\in\N^2}$ satisfies the 
hypotheses of Kingman' subadditive ergodic theorem 
(see \citealp[Theorem VI.2.6]{MR2108619}) by \eqref{eq:sous-additif}. 
Hence  (noticing also that $\vartheta_z(0,n)=\vartheta_{nz}(0,1)$) there exists $\mu(z)\in\R^+$ such that
\begin{eqnarray}\label{eq:lim-vartheta_theta}
\lim_{n\to +\infty}\frac 1n \vartheta_z(0,n)
&=&\lim_{n\to +\infty} E\left(\frac{\vartheta_z(0,n)}{n}\right)
=\lim_{n\to +\infty} E\left(\frac{\vartheta_{nz}(0,1)}{n}\right)\cr
&=&\inf_{n\in\N} E\left(\frac{\vartheta_z(0,n)}{n}\right)
=\inf_{n\in\N} E\left(\frac{\vartheta_{nz}(0,1)}{n}\right)
=\mu(z)\,
\end{eqnarray} 
 a.s. and in $L^1$.  Since the random variables  $(u(z):z\in \Z^d)$ are identically distributed,
it follows from Lemma \ref{lem:regularite-approx-tau} and Chebychev's inequality that \newline
$\sum_{n=0}^\infty P(u(nz)>n\eps)<+\infty$
 for all $\eps>0$, so that by Borel-Cantelli's Lemma
\begin{equation}\label{eq:lim-u}
\lim_{n\to +\infty}\frac {u(nz)}n =0,\,\mbox{ a.s.} 
\end{equation}
 Thus by  \eqref{eq:vartheta_theta}, \eqref{eq:lim-vartheta_theta}, \eqref{eq:lim-u}  
 we have \eqref{eq:lim-widehat_tau} for all $z\in\Z^d$.\\ \\ 
\textit{(ii)} Since $ R_o^{\rm out}$ is a.s. finite,  if $nz\in C_o^{\rm out}$, then   
$nz\in C_o^{\rm out}\setminus R_o^{\rm out}$ for $n$ large enough. Hence,  from 
 Lemma \ref{lem:comparaison-approximation}, for $n$ large enough we have
\[
\left|\dsp{\frac{\tau(o,nz)}n}-\mu(z)\right|{\bf 1}_{\{nz\in C_o^{\rm out}\setminus R_o^{\rm out}\}}
 \le \dsp{\frac{u(o) + u(nz)}n}+\left|\dsp{\frac{\widehat{\tau}(o,nz)}n}-\mu(z)\right|
\]
and we conclude  that \eqref{eq:radial-limit_2} is satisfied  by \eqref{eq:lim-u} and \eqref{eq:lim-widehat_tau}. 
\end{proof}
 \subsection{Extending $\mu$}\label{subsec:extending_mu}
 We have proved the existence of a linear propagation speed 
in every direction of $\Z^d$.
Now, to derive an asymptotic shape result, in particular 
for the approximating  travel  times
$(\widehat\tau(x,y),x,y\in\Z^d)$, we need to 
 extend $\mu$ from $\Z^d$ to 
 a Lipschitz, convex and homogeneous function on $\R^d$. 
The asymptotic shape of the epidemic will be given
by the convex set $D$ defined   in \eqref{def:D} below. 
As a first step, we prove properties
of $\mu$ on $\Z^d$. 
\begin{lemma}\label{ajoute1}
 The function $\mu$ satisfies the following properties for all $x,y\in \Z^d$, $k \in \N$:
 \newline
(i)  $\displaystyle{\mu(x)=\lim_{n\to+\infty} E\left(\frac{\widehat \tau (o,nx)}{n}\right)}$,\newline
(ii) $\mu(x+y)\leq \mu(x)+\mu(y)$,\newline
(iii) $\mu(x)=\mu(-x)$,\newline
 (iv) $\mu({\rm e}_i)=\mu({\rm e}_\ell),\,\forall i,\ell\in \{1,\dots,d\}$,\newline
 (v) $\mu(kx)=k\mu(x)$,\newline
(vi) $\mu(x)\leq \mu({\rm e}_1)\|x\|_1$.
\end{lemma}
\begin{proof}{Lemma}{ajoute1}
Since $\vartheta_x(0,n)=\widehat{\tau}(o,nx)+u(nx)$, part \textit{(i)} 
follows from \eqref{eq:lim-vartheta_theta} and  \eqref{eq:lim-u}.
 To prove part \textit{(ii)} write:
\begin{eqnarray*}
&&\mu(x+y)=\lim_{n\to+\infty} E\left(\frac{\widehat \tau (o,n(x+y))}{n}\right)\\
&\leq& \lim_{n\to+\infty} E\left(\frac{\widehat \tau (o,nx)}{n}\right)
+ \lim_{n\to+\infty} E\left(\frac{\widehat \tau (nx,n(x+y))}{n}\right)
+\lim_{n\to+\infty} E\left(\frac{u(nx)}{n}\right)\\
&=& \lim_{n\to+\infty} E\left(\frac{\widehat \tau (o,nx)}{n}\right)
+ \lim_{n\to+\infty} E\left(\frac{\widehat \tau (nx,n(x+y))}{n}\right)\\
&=&\mu(x)+\mu(y),
\end{eqnarray*}
where the first equality follows from part \textit{(i)},  
the inequality from  \eqref{eq:sous-additif},  the second equality  from \eqref{eq:lim-u} 
 and the third one  from part \textit{(i)} and translation invariance of 
 $\widehat \tau$. Parts \textit{(iii)}--\textit{(iv)} follow immediately from 
 part \textit{(i)} and the corresponding properties of $\widehat \tau(o,x)$.
To prove part \textit{(v)} 
write:
$$\mu(kx)
=\lim_{n\to +\infty} E\left(\frac{\vartheta_{nkx}(0,1)}{n}\right)
=k\lim_{n\to +\infty} E\left(\frac{\vartheta_{nkx}(0,1)}{nk}\right)
=k\mu(x),$$
where the first and third equalities follow from \eqref{eq:lim-vartheta_theta}. 
Finally, part \textit{(vi)} follows from parts \textit{(ii)--(iv)}.
\end{proof}
\mbox{}\\ \\
 Next corollary extends Lemma \ref{ajoute1}\textit{(iii)}--\textit{(iv)}.
\begin{corollary}\label{cor:per}
 For any permutation
$\sigma$ of $\{1,\cdots,d\}$, any 
$y=(y_1, y_2,\cdots,y_d)\in\Z^d$ and any choice of the signs $\pm$,
\[ 
\mu(\pm y_{\sigma(1)}, \pm y_{\sigma(2)},\cdots,\pm y_{\sigma(d)})=\mu(y_1, y_2,\cdots,y_d).
\]
\end{corollary}

\begin{proof}{Corollary}{cor:per}
Clearly $\widehat \tau (o,(y_1,y_2,\dots,y_d))$ has the same distribution as 
$\widehat \tau(o,(\pm y_{\sigma(1)}, \pm y_{\sigma(2)},\cdots,\pm y_{\sigma(d)}))$ 
for any choice of the signs and any permutation $\sigma$, hence the corollary follows
 from  Lemma \ref{ajoute1}\textit{(i)}.
\end{proof}
\begin{lemma}\label{ajoute2}
Let $\gamma^*=\mu({\rm e}_1)$. Then $\gamma^*$ is a Lipschitz constant for $\mu$.
For all $u,v\in \Z^d$ we have
$$\vert \mu(u)-\mu(v)\vert \leq \gamma^* \|u-v\|_1. $$
\end{lemma}
\begin{proof}{Lemma}{ajoute2}
 Let $y=u-v,\,x=v$.  We have  
$$\mu(u)-\mu(v)=\mu(x+y)-\mu(x)\leq \mu(y)=\mu(u-v)\leq \mu({\rm e}_1)\|u-v\|_1,$$
where the inequalities follow from Lemma \ref{ajoute1}\textit{(ii)} and \textit{(vi)}.
Similarly, taking $x=u,\,y=v-x$ gives
$$\mu(v)-\mu(u)\leq \mu({\rm e}_1)\|v-u\|_1,$$
and the lemma follows.  
\end{proof}
\mbox{}\\ \\
\noindent
 In a second step, we extend $\mu$ to $\R^d$ and we introduce the set $D$.
\begin{proposition}\label{lem:definir_phi}
There exists an extension of $\mu$ to $\R^d$, which is 
Lipschitz  with Lipschitz constant $\gamma^*$ given by Lemma  \ref{ajoute2}, 
convex and homogeneous 
on $\R^d$.  Moreover,
$\mu(x) = 0$ if and only if $x=o$ and the set
\be\label{def:D} 
D=\{x\in\R^d:\mu(x)\le 1\}
\ee
is convex, bounded and contains an open ball centered at $o$.
\end{proposition}
\begin{proof}{Proposition}{lem:definir_phi}
We start by extending $\mu$ to $\Q^d$. For $x\in\Q^d\setminus\{o\}$ let 
\beq\label{eq:Nx}
N_x&=&\min\{k\ge 1,k\in \N :kx\in\Z^d\}\quad \hbox{  and}\\     
\label{def:bis-mu-x}\mu(x)&=&\frac{\mu(N_x x)}{N_x}.\eeq
We now prove that this extension is homogeneous: let $\alpha\in \Q$ be  positive and 
let $x\in \Q^d,\,x\neq o$. Then, there exist
$k_1,k_2\in \N$ multiples of $N_x$ and $N_{\alpha x}$ respectively, such that  
$k_1x,k_2\alpha x\in \Z^d$ and $k_1x=k_2\alpha x$.
Write
\begin{eqnarray*}
\mu(\alpha x)
=\frac{\mu(N_{\alpha x}\alpha x )}{N_{\alpha x}}=\frac{\mu(k_2 \alpha x)}{k_2}
=\frac{\mu(k_1 x)}{k_2}=\frac{k_1 }{k_2}\frac{\mu(k_1 x)}{k_1}=\alpha \frac{\mu(N_x x)}{N_x}=\alpha \mu(x),
\end{eqnarray*}
using \eqref{def:bis-mu-x} for the  first equality, Lemma \ref{ajoute1}\textit{(v)} 
for the  second and  fifth ones.\par
\noindent
  To prove that $\mu$ is Lipschitz on $\Q^d$, let $x,y\in \Q^d\setminus\{o\}$. Then,
 \begin{eqnarray*}
 \left\vert \mu(x)-\mu(y)\right\vert&=&\left\vert \frac{\mu(N_x x)}{N_x}-\frac{\mu(N_y y)}{N_y}\right\vert
=\left\vert \frac{\mu(N_yN_x x)}{N_yN_x}-\frac{\mu(N_xN_y y)}{N_xN_y}\right\vert\\
&=& \frac{\vert \mu(N_xN_y x)-\mu(N_xN_y y)\vert}{N_xN_y}\\
&\leq & \frac{\gamma^*  \|N_xN_y x-N_xN_y y\|_1}{N_xN_y}=\gamma^*  \| x- y\|_1,
\end{eqnarray*}
using Lemma \ref{ajoute1}\textit{(v)} for the second equality and Lemma \ref{ajoute2} for the inequality.\\ \\
To prove that $\mu$ is convex on  $\Q^d$, 
take $x,y \in \Q^d$ and $\alpha \in \Q \cap (0,1)$.
Then let $k_1,k_2$ be elements in $\N$ such that 
$k_1\alpha \in \N$, $k_2 x \in \Z^d$, $k_2 y \in \Z^d$ and write: 
\begin{eqnarray*}
&&\mu (\alpha x+(1-\alpha)y)=\lim_{n \to+\infty}E\Big( \frac{\widehat \tau(o,n\alpha x+n(1-\alpha)y)}{n}\Big)\cr
&=&\lim_{n \to+\infty} E\Big(\frac {\widehat \tau(o,nk_1\alpha k_2x+nk_1(1-\alpha)k_2y)}{nk_1k_2}\Big)\cr
&\leq&\lim_{n \to+\infty} E\Big(\frac {\widehat \tau(o,nk_1\alpha k_2x)+\widehat \tau (o,nk_1(1-\alpha)k_2y)+
u(nk_1\alpha k_2x)}{nk_1k_2}\Big)\cr
&=&\lim_{n \to+\infty} E\Big(\frac {\widehat \tau(o,nk_1\alpha k_2x)+\widehat \tau (o,nk_1(1-\alpha)k_2y)}{nk_1k_2}\Big)\cr
&=&\frac{\mu(k_1k_2\alpha x)+\mu(k_1k_2(1-\alpha)y)}{k_1k_2}\cr
&=&\alpha \mu(x)+(1-\alpha)\mu(y),
\end{eqnarray*}
where the first equality follows from Lemma  \ref{ajoute1}\textit{(i)},  the inequality from  Lemma \ref{lem:sous-additif}, the third equality
 from \eqref{eq:lim-u}, 
 the fourth 
from Lemma \ref{ajoute1}\textit{(i)} and the last one from the homogeneity of $\mu$ on $\Q^d$.\\ \\
Because $\mu$ is homogeneous, Lipschitz and convex on $\Q^d$, we
can extend $\mu$ by continuity to $\R^d$.\\ \\
To prove that $\mu(x)>0$ if $x\neq o$ we argue by contradiction: assume $\mu(x)=0$ and without loss of generality that
$x=(x_1,\dots,x_d)$ with $x_1\neq 0$. First note that since $\mu$ is Lipschitz and homogeneous, 
the conclusion of Corollary \ref{cor:per} also holds for any $(x_1,\dots,x_d)\in \R^d$,  then write
\begin{eqnarray*}
\mu(2x_1,0,\cdots,0)&=&\mu(2x_1,0,\cdots,0)- \mu(x_1,x_2,\cdots,x_d)\\
&\le&\mu(x_1,-x_2,\cdots,-x_d)=0,
\end{eqnarray*}
using Lemma \ref{ajoute1}\textit{(ii)} for the inequality, and 
Corollary \ref{cor:per} with the assumption $\mu(x) = 0$ for the last equality. 
Then since  $\mu$ is homogeneous we get $\mu({\rm e}_1)=0$. 
However, considering a standard first passage percolation model 
with passage times $e(z,y)$ and adding a `tilde' to 
quantities associated to this model, we have $\widetilde\tau(o,z)\leq \tau(o,z)$ a.s. for all  $z\in \Z^d$. 
Since by \cite[Theorem (2.18)]{MR0876084},
$$\lim_{n\to+\infty} \widetilde \tau (o,n{\rm e}_1)=\widetilde \mu ({\rm e}_1),$$ 
it follows from 
\eqref{eq:radial-limit_2}   that $\widetilde\mu({\rm e}_1)\leq \mu({\rm e}_1)=0$.
 But  from  \cite[Theorems (1.7) and (1.15)]{MR0876084}
we get $\widetilde\mu({\rm e}_1)>0$, thus reaching a contradiction.\\ \\
The convexity of $\mu$ implies that $D$ is convex. We prove by contradiction that
$D$ contains an open ball centered at $o$: otherwise, there exists a sequence $(x_n)_{n\in\N}$
such that $x_n\notin D,\,\lim_{n\to+\infty}x_n=0$; therefore on the one hand $\mu(x_n)>1$, and on the other
hand $\lim_{n\to+\infty}\mu(x_n)=0$ because $\mu(o)=0$ and $\mu$ is continuous, hence a contradiction.\\ \\
Finally we argue again by contradiction to prove that the set $D$ is bounded: 
otherwise  there would exist  a sequence $(y_n)_{n\in\N}$ with $y_n\in D$ and $\|y_n\|_1>n$.
Then  $x_n=y_n/\|y_n\|_1$ satisfies $\|x_n\|_1=1$, and, 
since  $\mu$ is homogeneous, 
 $\lim_{n\to+\infty}\mu(x_n)=0$. By compactness  $(x_n)_{n\in\N}$
 has a converging subsequence to some $x$ such that $\mu(x)=0$ with $\|x\|_1=1$; 
 since we have already proved there is no such $x$ 
we get a contradiction.
\end{proof}
\subsection{Behavior of $\widehat \tau$}\label{subsec:behavior-widehat_tau}
 Our next result establishes how $\widehat \tau(o,z)$ grows for $z\in\Z^d$.  
\begin{theorem}\label{th:widehat_t-serie}
There exist $K=K(\lambda,d)>0$ and $\alpha>0$ such that 
\begin{eqnarray*}
P(\widehat\tau(o,z)>K\|z\|_\infty)&\leq& \exp(-\alpha(\|z\|_\infty^{1/d}),\ \ \forall \ z\in \Z^d,\cr
P(\widehat\tau(o,z)>K(\|z\|_\infty+n))&\leq& \exp(-\alpha n^{1/d}),\ \ \forall \ z\in \Z^d,n\in \N,\cr
\sum_{z\in\Z^d} P({\widehat\tau}(o,z)>K\|z\|_\infty)&<&+\infty.
\end{eqnarray*} 
\end{theorem}
\begin{proof}{Theorem}{th:widehat_t-serie}
Let $K\ge 0, z\in\Z^d$ and let 
$\mathcal B=B(o,(\|z\|_\infty+n)/4)\times B(z,(\|z\|_\infty+n)/4)$. Then write:
\begin{eqnarray}\label{eq:calcul1-widehat_t-serie}
&&P({\widehat\tau}(o,z)>K(\|z\|_\infty+n))\cr
&\le & P(4\kappa(z)>\|z\|_\infty+n)+   P(4\kappa(o)>\|z\|_\infty+n)    
 + P(A)
\end{eqnarray}
where
\begin{eqnarray}\nonumber
A&=&\{{\widehat\tau}(o,z)>K(\|z\|_\infty +n), 4\kappa(z) 
\le \|z\|_\infty +n, 4\kappa(o) \le \|z\|_\infty +n\}\\\label{eq:calcul2-widehat_t-serie}
 &\subset& \displaystyle{\cup_{(x,y)\in \mathcal B}\{ x\to y, \tau(x,y)>K(\|z\|_\infty+n)\}}.
 \end{eqnarray}
Note that if $(x,y)\in \mathcal B$ we have 
\begin{eqnarray}\label{eq:x-y and z}
\|z\|_\infty-n&\le& 2\|x-y\|_\infty\le 3\|z\|_\infty+n\quad\mbox{ and}\cr
3(\|z\|_\infty+n)&=& 3\|z\|_\infty+n+2n\ge 2(\|x-y\|_\infty+n).
\end{eqnarray}
{}From \eqref{eq:calcul2-widehat_t-serie}, \eqref{eq:x-y and z}, 
for ${\rm C}_2$ given in Proposition \ref{lem:ap2}, we get:
\begin{eqnarray}\label{eq:calcul5-widehat_t-serie}
&&
P(A) \le
\sum_{x\in B(o,(\|z\|_\infty+n)/4)}\ \ \sum_{y\in B(z,(\|z\|_\infty+n)/4)} \cr
&&\Big(P( 3\tau(x,y)> 2K(\|x-y\|_\infty +n), D(x,y)<({\rm C}_2+1)(\|x-y\|_1+n))\cr
&&\quad+P(x\to y, D(x,y)\geq ({\rm C}_2+1)(\|x-y\|_1+n ))\Big).
\end{eqnarray}
It now follows from Proposition \ref{lem:ap2}\textit{(i)}  that 
we have, for some $\alpha_5>0$,
\begin{eqnarray}\label{eq:calcul6-widehat_t-serie}
 &&P(x\to y, D(x,y)\geq ({\rm C}_2+1)(\|x-y\|_1+n ))\cr
 &\le& \exp(-\alpha_5(\|x-y\|_1+n)^{1/d})\cr
 &\le& \exp(-\alpha_5(\|x-y\|_\infty+n)^{1/d}).
\end{eqnarray} 
Then, taking $K$ large enough, by large deviation results for exponential
variables,  we also have, for some $\alpha_6>0$,
\begin{eqnarray}\label{eq:calcul7-widehat_t-serie}
&& P( 3\tau(x,y)>2K(\|x-y\|_\infty +n), D(x,y)<({\rm C}_2+1)(\|x-y\|_1+n)) \cr
&\le& P( 3\tau(x,y)>2K(\|x-y\|_\infty +n), D(x,y)<({\rm C}_2+1)d(\|x-y\|_\infty+n))\cr
&\le& \exp(-\alpha_6(\|x-y\|_\infty+n)).
\end{eqnarray}
Hence, from  \eqref{eq:x-y and z}--\eqref{eq:calcul7-widehat_t-serie}, for some constants $R$
and $\alpha_7>0$ we have:
\begin{eqnarray*}\label{eq:calcul9-widehat_t-serie}
&&P(A) \le R (\|z\|_\infty+n)^{2d}\exp(-\alpha_7 (\|z\|_\infty+n)^{1/d}),
\end{eqnarray*} 
which gives, by modifying the constants, 
\begin{eqnarray}\label{eq:calcul10-widehat_t-serie}
&&P(A) \le R' \exp(-\alpha_8 (\|z\|_\infty+n)^{1/d}).
\end{eqnarray} 
The theorem's statements now follow from \eqref{eq:calcul9-widehat_t-serie}, 
\eqref{eq:calcul1-widehat_t-serie} and Lemma \ref{lem:k_x-exp_tail}.
\end{proof}
\subsection{Asymptotic shape for $\widehat \tau$}\label{subsec:shape-widehat_t}
 Next theorem is the last necessary step to prove the shape theorem. 
\begin{theorem}\label{th:At-encadre}
Let $\eps>0$, and
${\widehat A}_t=\{z\in\Z^d:{\widehat\tau}(o,z)\le t\}.$
Then, a.s. for $t$ large enough, for $D$ defined in \eqref{def:D},
\begin{equation}\label{eq:At-encadre}
(1-\eps)tD\cap\Z^d\subset{\widehat A}_t\subset(1+\eps)tD\cap\Z^d.
\end{equation}
\end{theorem}
In the sequel $K$ and $\alpha$ are fixed constants satisfying the conclusions of 
Theorem \ref{th:widehat_t-serie}, $\gamma^*$ is the Lipschitz constant 
of $\mu$ (see Lemma \ref{ajoute2}) and 
$N_x$ was defined in \eqref{eq:Nx} for any $x\in \Q^d\setminus \{o\}$.
 To prove Theorem \ref{th:At-encadre} we need the two following lemmas.  
\begin{lemma}\label{lem:covering B_0}
Let $\rho>0$ and let $\delta \le \rho/(2K)$. Then,  for all $x\in \Q^d\setminus \{o\}$, 
\begin{eqnarray}\label{eq:f1}
\sum_{k>0}P(\sup_{z\in B_{kN_xx}(\delta kN_x)\cap\Z^d}\widehat \tau(kN_xx,z)\ge kN_x\rho)&<&\infty,\\
\label{eq:f2}
\sum_{k>0}P(\sup_{z\in B_{kN_xx}(\delta kN_x)\cap\Z^d}\widehat \tau(z,kN_xx)\ge kN_x\rho)&<&\infty.
\end{eqnarray}
\end{lemma}
\begin{proof}{Lemma}{lem:covering B_0}
We derive only  \eqref{eq:f1}, since the proof of \eqref{eq:f2} is analogous.
By translation invariance 
$$P(\sup_{z\in B_{kN_xx}(\delta kN_x)\cap \Z^d}\widehat \tau(kN_xx,z)\ge kN_x\rho)
=P(\sup_{z\in B(\delta kN_x)\cap \Z^d}\widehat \tau(o,z)\ge kN_x\rho).$$
Hence it suffices to show that
$$\sum_{k>0} P(\sup_{z\in B(\delta kN_x)\cap \Z^d}\widehat \tau(o,z)\ge kN_x\rho)<\infty.$$
 Let  $k>0,z\in B(\delta kN_x)\cap\Z^d$. By Theorem \ref{th:widehat_t-serie} 
 we have: 
\begin{eqnarray*}
P(\widehat \tau (o,z)\ge kN_x\rho)&\le& P(\widehat \tau (o,z)\ge K\|z\|_\infty +kN_x\rho/2)\cr
&\le& \exp\Big(-\alpha \Big\lfloor \frac{kN_x\rho}{2K}\Big\rfloor^{1/d}\Big).
\end{eqnarray*}
 Therefore, for some constant ${\rm C}$,
\begin{eqnarray*}
\sum_{k>0}P(\sup_{z\in B(\delta kN_x)\cap\Z^d}\widehat \tau(o,z)\ge kN_x\rho)
&\le& \sum_{k>0} {\rm C} (\delta kN_x)^d \exp\Big(-\alpha \Big\lfloor \frac{kN_x\rho}{2K}\Big\rfloor^{1/d}\Big)\cr
&<&\infty.
\end{eqnarray*}
\end{proof}
\mbox{}\\ \\
For $x\in\Q^d\setminus\{o\}$, $\delta>0$, we define
the cone  associated to $x$ of amplitude $\delta$ as 
\begin{equation}\label{eq:cone_de_x}
C_x(\delta)=\Z^d\cap \Big(\cup_{t\ge 0}B_{tx}(\delta t)\Big).
\end{equation} 
\begin{lemma}\label{lem:included-cone}
 Let $x\in \Q^d\setminus \{o\}$. Then for any $0<\delta' <\delta$ 
 the set $C_x(\delta')\setminus \cup_{k\ge 0} B_{kN_x x}(\delta k N_x)$ is finite.
\end{lemma}
 \begin{proof}{Lemma}{lem:included-cone}
Let
 $$t_0=\frac{N_x\|x\|_1}{\delta-\delta'}.$$ 
 Since 
$\Z^d\cap \Big(\cup_{t\ge 0}^{t_0}B_{tx}(\delta' t)\Big)$ is finite, it suffices to show that
$$\cup_{t\ge t_0}^{\infty}B_{tx}(\delta' t)
\subset \cup_{k\ge 0} B_{kN_x x}(\delta k N_x).$$
To prove this, pick
$z\in B_{t_1x}(\delta' t_1)$ for some $t_1\ge t_0$. Let $k_0=\inf\{i\in \N: iN_x\ge t_1\}$.
Hence $0\le k_0N_x-t_1<N_x$, and
\begin{eqnarray*}
\|z-k_0N_xx\|_1&\le&\|z-t_1x\|_1+|t_1-k_0N_x|\|x\|_1\\
&<& \|z-t_1x\|_1+N_x\|x\|_1
\le \delta't_1 +N_x\|x\|_1\\
&=&\delta't_1+(\delta-\delta')t_0\leq \delta t_1\leq \delta k_0 N_x.
\end{eqnarray*}
Therefore $z\in  B_{k_0 N_x x}(\delta k_0 N_x)$
and the lemma is proved. 
\end{proof} 
\mbox{}\\ \\
 In the next proof we use that since the 
 Lipschitz constant of  $\mu$ is $\gamma^*$ 
 for the norm $\|.\|_1$
(by Proposition \ref{lem:definir_phi}), 
it is $\gamma=\gamma^*d$ for the norm $\|.\|_\infty$.\par
\medskip
\begin{proof}{Theorem}{th:At-encadre}
Fix $\eps\in (0,1)$ and let $\rho,\delta $ and $\iota$ be three small 
positive  parameters  such that $\delta \le \rho/(2K)$, whose values  
will be determined later. 
The set ${\mathcal Y}=\{x\in \Q^d:1-2\iota <\mu(x)<1-\iota\}$ is a ring between two
balls with the same center but with  a different radius, because 
by Proposition \ref{lem:definir_phi},
 $\mu$ 
is homogeneous and  positive except that $\mu(o)=0$. Hence the (compact) closure of ${\mathcal Y}$,
which is recovered
by balls of the same radius centered on the rational points of ${\mathcal Y}$,
is in fact covered by a finite number of such balls. 
Thus there exists a finite subset $Y$ of ${\mathcal Y}$
 such that
$\Z^d \subset \cup_{x\in Y}C_x(\delta/2)$ (if the balls recover the ring, 
the cones associated to them recover the whole space). Hence, to prove the first 
inclusion of \eqref{eq:At-encadre} it suffices to show that for any 
$x\in Y$ and any sequences $(t_n)_{n>0}$ and 
$(z_n)_{n>0}$ such that $t_n \uparrow \infty$ in
$\R^+$, 
 $z_n\in C_x(\delta/2)\cap\Z^d$  with $\|z_n\|_\infty\ge n$ and 
 $\mu(z_n)\le (1-\eps) t_n$, we have $\widehat \tau (o,z_n)\le t_n$ 
 a.s. for $n$ sufficiently large. So, let $(t_n)_{n>0}$ and 
$(z_n)_{n>0}$ be such sequences. Using Lemma \ref{lem:included-cone}
 (taking a subsequence if necessary) 
let $k_n\in \N $ be such that $z_n\in B_{k_nN_xx}( \delta k_nN_x)$, hence 
$k_n\ge {\rm C}n$ for some constant ${\rm C}$.  Since 
 $\mu$ is Lipschitz,  write  
$$k_n N_x (1-2\iota)\le \mu(k_nN_xx)\le \mu(z_n)
+\gamma \delta k_n N_x\le (1-\eps)t_n+\gamma \delta k_n N_x, $$
so that
\be\label{eq:this inequality}
k_nN_x\le \Big(\frac{1-\eps}{1-2\iota -\gamma \delta}\Big)t_n.
\ee
It now follows from \eqref{eq:this inequality} and the subadditivity property 
\eqref{eq:sous-additif} of $\widehat \tau$ that:
$$\frac{\widehat \tau(o,z_n)}{t_n}\le  
\Big(\frac{1-\eps}{1-2\iota 
-\gamma \delta}\Big)\Big(\frac{\widehat \tau(o,k_nN_xx)}{k_nN_x}
+\frac{u(k_nN_xx)}{k_nN_x}
+\frac{\widehat \tau(k_nN_xx,z_n)}{k_nN_x}\Big).$$
Therefore, by Theorem \ref{th:radial-limits}, 
Lemma \ref{lem:regularite-approx-tau} 
(the variables $u(.)$ are identically distributed, and 
$k_n\ge {\rm C}n$), Lemmas \ref{lem:definir_phi} and  
\ref{lem:covering B_0} we obtain:
$$\limsup_{n\to +\infty}  \frac{\widehat \tau(o,z_n)}{t_n}\le 
\Big(\frac{1-\eps}{1-2\iota -\gamma \delta}\Big)\Big(\mu(x)
+\rho \Big)\qquad \mbox{a.s.}$$
Since $x\in Y$ this implies: 
$$\limsup_{n\to +\infty}  \frac{\widehat \tau(o,z_n)}{t_n}
\le \Big(\frac{1-\eps}{1-2\iota -\gamma \delta}\Big)\Big(1-\iota+\rho \Big)\qquad \mbox{a.s.}$$ 
 Taking $\iota$, $\rho$ and $\delta$ small enough, the right hand 
side is strictly less than $1$ which proves that $\widehat \tau (o,z_n)\le t_n$ 
a.s. for $n$ sufficiently large.

Similarly, to prove the second inclusion of \eqref{eq:At-encadre} it suffices 
to show that for any $x\in Y$ and any sequences $t_n \uparrow \infty$ in
$\R^+$ and 
 $z_n$  in $C_x(\delta/2)\cap\Z^d$  such that $\mu(z_n)\ge (1+\eps) t_n$ we have 
 $\widehat \tau (o,z_n)> t_n$ a.s. for $n$ sufficiently large. As before,  taking subsequences if necessary, 
 we let 
$(t_n)_{n>0}$ and 
$(z_n)_{n>0}$ be such sequences, and $k_n\in \N $ be such that 
$z_n\in B_{k_nN_xx}( \delta k_nN_x)$. Proceeding then as for the first inclusion, we get:
$$k_nN_x(1-\iota)\ge \mu(k_nN_xx)\ge \mu(z_n)-\gamma \delta k_nN_x
\ge (1+\eps)t_n -\gamma \delta k_nN_x,$$
$$k_nN_x\ge \Big( \frac{1+\eps}{1-\iota +\gamma \delta}\Big)t_n,$$
$$\frac{\widehat \tau(o,z_n)}{t_n}\ge \Big(\frac{1+\eps}{1-\iota 
+\gamma \delta}\Big) \Big( \frac{\widehat \tau (o,k_nN_xx)}{k_nN_x}
-\frac{u(z_n)}{k_nN_x}-\frac{\widehat \tau (z_n,k_nN_xx)}{k_nN_x}\Big),$$
and 
\begin{eqnarray*}
\liminf_{n\to +\infty} \frac{\widehat \tau(o,z_n)}{t_n}
&\ge& \Big(\frac{1+\eps}{1-\iota +\gamma \delta}\Big) \Big(\mu(x) -\rho\Big)\qquad \mbox{a.s.}\cr
&\ge&\Big(\frac{1+\eps}{1-\iota +\gamma \delta}\Big) \Big(1-2\iota-\rho\Big)\qquad \mbox{a.s.}
\end{eqnarray*}
 Now, taking $\iota$, $\rho$ and $\delta$ small enough, 
the right hand side is strictly bigger than $1$ and the 
second inclusion of \eqref{eq:At-encadre} is proved.
\end{proof}
\subsection{  Asymptotic shape for the epidemic  }\label{subsec:shape-thm}
We can now prove our main result,  the shape theorem. 
\begin{proof}{Theorem}{th:shape}
  Let $\eps>0$ be given. \par
\noindent
\textit{(i)} We first show that  
the infection grows at least linearly as $t$ goes to
infinity, that is, 
\begin{equation}\label{eq:ne-stagne-pas1}
P\Big(\big(\Upsilon_t\cup \Xi_t\big) \supset \big((1-\eps)tD\cap C_o^{\rm out} \big)\mbox{  for all $t$ large enough} \Big)=1.
\end{equation}
 Since $R_o^{\rm out}$ is finite a.s. this will follow from:
\begin{equation}\label{eq:ne-stagne-pas2}
P\Big(\big(\Upsilon_t\cup \Xi_t \big)\supset \big((1-\eps)tD\cap (C_o^{\rm out} \setminus R_o^{\rm out})\big)\mbox{  for all $t$ large enough} \Big)=1.
\end{equation}
By Theorem \ref{th:At-encadre}, if $0<a<b$ then $atD\cap \Z^d \subset \widehat A_{bt}$ a.s. for $t$ large enough. Hence, for  $t$ large enough
 $z\in (1-\eps)tD\cap (C_o^{\rm out} \setminus R_o^{\rm out})$ implies  
\be\label{eq:tau_z-borne}
 {\widehat\tau}(o,z)\le (1-\eps/2)t\mbox{ a.s.}
\ee
 and by Lemma \ref{lem:comparaison-approximation},
${\tau}(o,z)\le(1-\eps/2)t+u(o)+u(z)$. 
Since  $u(o)< +\infty$ a.s. we have $u(o)<(\eps/4)t$ a.s. for $t$ large enough. Hence,
 by \eqref{eq:xi-zeta-tau}, 
\eqref{eq:ne-stagne-pas2} will follow if we show that 
 $\sup_{z \in tD}u(z)\leq (\eps/4) t\,$ a.s.  for $t$ large enough, 
which is implied by
$\sup_{z \in (n+1)D}u(z)\leq (\eps/4)n\,$ a.s. for $n=\lfloor t \rfloor$. 
By Proposition \ref{lem:definir_phi}, $D$ is bounded, hence the number of points in  $(n+1)D$ with
coordinates in $\Z$ is less than ${\rm C_5}(n+1)^d$ for some constant ${\rm C_5}$. 
Then  write
\begin{eqnarray*}
P\left(\sup_{z \in (n+1)D}u(z)\geq \frac{\eps n}4\right)&\leq&
{\rm C_5}(n+1)^d P\left(u(o)\geq \frac{\eps n}4\right)\cr&\leq&
{\rm C_5}(n+1)^d \frac{4^{d+2}}{(\eps n)^{d+2}}E(u(o)^{d+2}).
\end{eqnarray*}
Thus, by Lemma \ref{lem:regularite-approx-tau},
$\sum_{n\in\N} P(\sup_{z \in (n+1)D}u(z)\geq \eps n/4)<\infty$, 
and \eqref{eq:ne-stagne-pas2} follows from 
Borel-Cantelli's Lemma.\\ \\
\textit{(ii)} Next we show that 
\begin{equation}\label{eq:ne-deborde-pas}
P\Big(\big(\Upsilon_t\cup \Xi_t\big) \subset\big( (1+\eps)tD\cap C_o^{\rm out}\big) \mbox{  for all $t$ large enough} \Big)=1.
\end{equation}
If $z$ belongs to $\Xi_t$ or $\Upsilon_t$, then 
by  \eqref{eq:xi-zeta-tau} and  Lemma \ref{lem:comparaison-approximation}, 
${\widehat\tau}(o,z)\le t$ for $z\in C_o^{\rm out}\setminus R_o^{\rm out}$, which implies 
$z\in(1+\eps)tD$ for $t$ large enough by Theorem \ref{th:At-encadre}. Since $R_o^{\rm out}$ is finite \eqref{eq:ne-deborde-pas} follows.\par
 Putting together \eqref{eq:ne-stagne-pas1} and \eqref{eq:ne-deborde-pas} yields \eqref{eq:shape}.\\ \\
\textit{(iii)} Finally, assuming $E(\vert T_z\vert^d)<\infty$, we show that
\begin{equation}\label{eq:brule-rapidement}
P(\Upsilon_t\cap(1-\eps)tD=\emptyset \mbox{  for $t$ large enough})=1.
\end{equation}
Let $z\in(1-\eps)tD\cap C_o^{\rm out}$, then, by 
\eqref{eq:xi-zeta-tau}, \eqref{eq:ne-stagne-pas1} and the same reasoning as for \eqref{eq:tau_z-borne}, 
we have $\tau(o,z)\le (1-\eps/2)t$  if $t$ is large enough.
Hence, \eqref{eq:brule-rapidement} will follow if we show that 
$T_z\ge (\eps/2)\tau(o,z)$ occurs only for a finite number of $z$'s.
Indeed otherwise $T_z\le (\eps/2)(1-\eps/2)t$ so that $\tau(o,z)+T_z<t$ if $t$ is large enough: it means that
the infection has reached
site $z$ and the time of infection from $z$ is over
before time $t$, hence $z$ has recovered by time $t$, 
that is $z\in\Xi_t,z\notin\Upsilon_t$. \par
   But 
for $\delta=(2(1+\eps)\sup_{x\in D}\|x\|_\infty)^{-1}$ (by Proposition \ref{lem:definir_phi}, $D$ is bounded),
 we have $\tau(o,z) \ge \delta \|z\|_\infty$ except for a finite number of $z$'s.
 Because if $z$ satisfies $\tau(o,z) < \delta \|z\|_\infty$, then by 
\eqref{eq:xi-zeta-tau} and \eqref{eq:ne-deborde-pas}, for $\delta \|z\|_\infty$ larger than some $t_0$,
we have
$z\in\big(\Upsilon_{\delta \|z\|_\infty}\cup \Xi_{\delta \|z\|_\infty}\big)
\subset (1+\eps)\delta \|z\|_\infty D$,
hence the contradiction $\|z\|_\infty\le \|z\|_\infty/2$. \par
Therefore, it suffices to show that for any $\delta'>0$ the event 
$\{T_z\ge \delta'\|z\|_\infty\}$ can only occur for a finite number of $z$'s. 
This will follow from Borel-Cantelli's Lemma
once we prove that $\sum_{z\in\Z^d} P(T_z\ge \delta' \|z\|_\infty)<\infty$. 
To do so we write, since the $T_z$'s are identically distributed:
$$\sum_{z\in\Z^d} P(T_z\ge \delta' \|z\|_\infty)
=\sum_{n\in\N}\sum_{z:\|z\|_\infty=n}P(T_z\ge \delta' n)
\le c\sum_{n\in\N}  n^{d-1}P(T_o\ge \delta' n)$$
for some constant $c$, and this last series converges because  
$T_o$ has a finite moment of order $d$. 
 Putting together \eqref{eq:shape} and \eqref{eq:brule-rapidement} yields \eqref{eq:couronne}.  
\end{proof}
\begin{appendix}\section{}\label{sec:appendix}
In this appendix we prove  Theorem \ref{prop:clusters_rentrant-sortant_infinis}, 
Lemma \ref{lem:connections} and \eqref{eq:analogue_G-thm8.21} in the proof
of Lemma \ref{lem:exp_decay_R}. 
 These proofs rely on dynamic renormalisation techniques introduced in \cite{MR1124831}. 
In applying these techniques we follow \cite[Chapter 7]{MR1707339} and \cite{MR1068308}, 
but we introduce some modifications. 
 In particular by considering some larger boxes we avoid using the sprinkling technique 
  more than once on any given bond. Because of this we only need to consider two different
   values of the infection parameter   and we do not need to introduce 
   the updating functions of \cite{MR1707339}.
   To simplify the notation we
write the proofs for $d=3$, but their generalizations to higher dimensions presents no problems.\\ \\
We introduce parameters whose values will be settled in Lemma \ref{lsuppl} below.
We fix $\lambda'>\lambda_c$ and 
adopt the terminology of \cite[Chapter 7]{MR1707339}. 
Nonetheless, we might change names of constants if this creates confusions
with the rest of our paper.
In the sequel $n,m$ and $N$ are positive integers such that 
\begin{equation}\label{mnN}
2m<n\qquad\hbox{  and  }\qquad N=n+m+1.
\end{equation}
 We consider our percolation model on  the slab $\Z^2 \times [-3N,3N]$. 
 Recall that we denote by $({\rm e}_1,{\rm e}_2,{\rm e}_3)$ the canonical basis of $\Z^3$.
For  $x=(x_1,x_2,x_3)\in \Z^3$ and $k\in \N$   such that 
$-3N+k\leq x_3\leq 3N-k$  we recall that $B(k)=[-k,k]^3$ and $B_x(k)=x+[-k,k]^3$.  
We divide the face $F(n)=\{x: x\in\partial B(n), x_1=n\}$ of $\partial B(n)$
in 4 quadrants:
\begin{eqnarray*}
T^{+,+}(n)&=&\{x: x\in\partial B(n), x_1=n, x_2\geq 0, x_3\geq 0\},\\
T^{+,-}(n)&=&\{x: x\in\partial B(n), x_1=n, x_2\geq 0, x_3\leq 0\},\\
T^{-,+}(n)&=&\{x: x\in\partial B(n), x_1=n, x_2\leq 0, x_3\geq 0\},\\
T^{-,-}(n)&=&\{x: x\in\partial B(n), x_1=n, x_2\leq 0, x_3\leq 0\}.
\end{eqnarray*}
and, for any choice of $(i,j)\in \{+,-\}^2$, we define a box 
 of thickness $2m+1$ composed of translates
of the corresponding quadrant, by 
\begin{equation}\label{eq:def-Tijmn}
 T^{i,j}(m,n)=\cup_{\ell=1}^{2m+1}\{\ell {\rm e}_1+T^{i,j}(n)\}.
 \end{equation}
For any of the above sets a subindex as $y$ means we translate it by $y$. 
\begin{definition} A {\rm seed} is a translate of $B(m)$ such that 
all its edges are $\lambda'$-open.
\end{definition}
We will be looking for oriented open paths starting 
in a seed inside $B(2N)$, and: either \textit{(a)}  contained in the union of boxes 
$B(3N)\cup B_{6N{\rm e}_1}(3N)$ and reaching a seed 
inside $B_{6N{\rm e}_1}(2N)\cap B_{8N{\rm e}_1}(2N)$;
 or \textit{(b)} contained in the union of boxes $B(3N)\cup B_{6N{\rm e}_2}(3N)$  
and reaching a seed inside $B_{6N{\rm e}_2}(2N)\cap B_{8N{\rm e}_2}(2N)$.
We will construct those paths in Lemma \ref{l3} below.  \\ \\
An important tool in  \cite[Chapter 7]{MR1707339}
is the \textit{sprinkling technique}, which enables some bonds, that would be closed 
otherwise, to be independently
open with a probability larger than some $\eps'>0$. We therefore need to find 
a way to proceed similarly, in spite of the fact that we work with
dependent percolation.   
For this,   it is convenient to define the processes for 
different values of  the rate of propagation $\widetilde\lambda$ 
on our common probability space and compare them. 
 Let  $(e_1(x,y),x,y\in\Z^3)$  be a collection of independent 
 exponential r.v.'s 
with parameter  $1$. Then let $e_{\widetilde\lambda}(x,y)={\widetilde\lambda}^{-1}e_1(x,y)$,
and  
\begin{equation}\label{def:sprinkling_open-closed-bonds}
X_{\widetilde\lambda}(x,y)=
\begin{cases}
1 &\text{if $e_{\widetilde\lambda}(x,y)<T_x$;}\\
0 &\text{otherwise.}
\end{cases}
\end{equation}
 We recall that the event  $\{T_x>e_{\widetilde\lambda}(x,y) \}$ occurs if and only if 
the oriented bond $(x,y)$ is $\widetilde\lambda$-open.  \\ \\ 
 The following lemma implies  that given  $\lambda>\delta_1>0$, 
 there exists $\iota>0$ such that for any
$\widetilde\lambda$ such that $\widetilde\lambda+\delta_1<\lambda$ the random field 
$\{X_{\widetilde\lambda +\delta_1}(u,v):u,v \in \Z^3\}$ is 
stochastically above the random field
$\{\max \{X_{\widetilde\lambda }(u,v),Y(u,v)\}:u,v \in \Z^3\}$ 
where the random variables $Y(u,v)$ are i.i.d. Bernoulli 
with parameter  $\iota$  and are independent of the random 
variables $X_{\widetilde\lambda}(u,v)$.
This lemma justifies the use of the sprinkling technique
 needed to prove  Lemmas \ref{l2} and \ref{l3} below. 
\begin{lemma}\label{lem:sprinkling} 
 Assume  $\lambda>\delta_1>0$. 
There exists  $\iota>0$  such that for any $\widetilde\lambda>0$ such that
$\widetilde\lambda+\delta_1<\lambda$, and any $x,y\in \Z^3,$ with $y\sim x$, 
\begin{eqnarray}\nonumber
P(X_{\widetilde\lambda+\delta_1 } (x,y)=1&\vert&   X_{\widetilde\lambda+\delta_1}(u,v),\,
u,v\in \Z^3,\, u\sim v, (u,v)\not=(x,y);\\\nonumber
&&X_{\widetilde\lambda}(u,v),\, u,v\in \Z^3,\, u\sim v)>\iota\ \ \ a.s.    
\end{eqnarray}
\end{lemma} 
\begin{proof}{Lemma}{lem:sprinkling}
Recall that $X_{\widetilde\lambda+\delta_1}(x,y)$ is independent of the random variables 
$X_{\widetilde\lambda+\delta_1}(u,v),X_{\widetilde\lambda}(u,v),u\neq x,v\sim u$, 
hence it suffices to show that
\begin{eqnarray}\label{(3)bis}
P(X_{\widetilde\lambda+\delta_1}(x,y)=1&\vert& X_{\widetilde\lambda+\delta_1}(x,z),\ z\sim x, 
 z\neq y;\nonumber\\ && X_{\widetilde\lambda}(x,z),z\sim x)>
    \iota \ \ \ a.s. 
    \end{eqnarray}
  Since we are now conditioning on a finite number of random variables taking only values $0$ and $1$, 
  \eqref{(3)bis} will follow from:
\begin{eqnarray}
P(X_{\widetilde\lambda+\delta_1}(x,y)=1&\vert& X_{\widetilde\lambda+\delta_1}(x,z)=a_z, 
z\sim x,  z\neq y;\nonumber\\ &&
 X_{\widetilde\lambda}(x,z)=b_z,z\sim x)> \iota\label{ex-lem3.6}
\end{eqnarray}
  for all choices $a_z$ and $b_z$ in $\{0,1\}$ with $b_z\leq a_z$.\\ \\
To prove \eqref{ex-lem3.6}, we denote by ${\mathcal N}_{x}$
the union of all partitions of the set 
$\{z\in \Z^3:\,z\sim x\}$ of neighbors of $x$ into three disjoint sets called
${\mathcal N}_{x}^1,{\mathcal N}_{x}^0,{\mathcal N}_{x}^{0,1}$ such that
$y\in{\mathcal N}_{x}^1\cup{\mathcal N}_{x}^{0,1}$. Using the inequality $P(A\vert B)\geq P(A\cap B)$ 
for two events $A,B$, and taking an arbitrary partition in ${\mathcal N}_{x}$ gives
  \begin{eqnarray*}
&&P(X_{\widetilde\lambda+\delta_1}(x,y)=1\,\vert\, X_{\widetilde\lambda+\delta_1}(x,u)=0,\, 
\forall u\in{\mathcal N}_{x}^0,\\
&&\quad X_{\widetilde\lambda}(x,v)=1,\, 
\forall v\in{\mathcal N}_{x}^1 ;  X_{\widetilde\lambda}(x,w)=0, X_{\widetilde\lambda+\delta_1}(x,w)=1,\,
\forall w\in{\mathcal N}_{x}^{0,1})\\
&&\ge
P(X_{\widetilde\lambda+\delta_1}(x,y)=1,\, X_{\widetilde\lambda+\delta_1}(x,u)=0,\, 
\forall u\in{\mathcal N}_{x}^0,\\
&&\quad X_{\widetilde\lambda}(x,v)=1,\, 
\forall v\in{\mathcal N}_{x}^1 ;  X_{\widetilde\lambda}(x,w)=0, X_{\widetilde\lambda+\delta_1}(x,w)=1,\,
\forall w\in{\mathcal N}_{x}^{0,1}).
 \end{eqnarray*}
Let $a>0$ be such that $P(T_{x} \in [a, a +\gamma ])>0$ 
for all $\gamma >0$,  and  let $\delta_2\in (0,a)$  be such that  $b$ defined by 
$$b:=\frac{(a+\delta_2)\widetilde\lambda}{\widetilde\lambda+\delta_1}$$ satisfies $b< a$. 
On the event
\begin{eqnarray}\nonumber
&&\{T_{x} \in [a,a+\delta_2/2),  e_{\widetilde\lambda} ({x},w) \in [a+\delta_2/2, a+\delta_2),\, 
\forall w\in{\mathcal N}_{x}^{0,1},\\\nonumber
&&\,e_{\widetilde\lambda} ({x},v) \in [a-\delta_2/2, a),\, 
\forall v\in{\mathcal N}_{x}^1,\\\label{eq:onthisevent}
&&e_{\widetilde\lambda+\delta_1} ({x},u) \in [a+\delta_2/2, a+\delta_2),\, 
\forall u\in{\mathcal N}_{x}^0\},
\end{eqnarray} 
for all sites $v$ such that $v\in{\mathcal N}_{x}^1$ we have  $e_{\widetilde\lambda}(x,v)<T_{x}$
hence $X_{\lambda'}(x,v)=1$; for all sites $u$ such that $u\in{\mathcal N}_{x}^0$ 
we have  $T_{x}<e_{\lambda'+\delta_1}(x,u)$
hence $X_{\widetilde\lambda+\delta_1}(x,u)=0$; and
for all sites $w$ such that $w\in{\mathcal N}_{x}^{0,1}$ we have  $T_{x}< e_{\widetilde\lambda}(x,w)$ and
\[e_{\widetilde\lambda+\delta}(x,w)= 
\frac{\widetilde\lambda}{\widetilde\lambda+\delta_1} e_{\widetilde\lambda}(x,w) 
 \leq \frac{\widetilde\lambda}{\widetilde\lambda+\delta_1}(a+\delta_2)=b<a \leq T_{x}\]
hence $X_{\widetilde\lambda}(u,w)=0$ and $X_{\widetilde\lambda+\delta_1}(u,w)= 1$. \par
Since the probability $p({\mathcal N}_{x}^0,{\mathcal N}_{x}^1,{\mathcal N}_{x}^{0,1})$ 
of the event \eqref{eq:onthisevent} is strictly positive,
 we conclude the proof of \eqref{ex-lem3.6} by taking
$$\iota=\inf_{{\mathcal N}_{x}}p({\mathcal N}_{x}^0,{\mathcal N}_{x}^1,{\mathcal N}_{x}^{0,1}).$$
\end{proof}
\mbox{}\\ \\
It is in view of this sprinkling procedure that we chose a propagation rate
$\lambda'>\lambda_c$. 
As in \cite[Section 7.2]{MR1707339}, we go on with two key geometrical lemmas. 
The first one, Lemma \ref{l1}, corresponds to \cite[Lemma 7.9]{MR1707339}, 
with a very similar proof that we omit consequently. The second one,
Lemma \ref{l2}, corresponds to \cite[Lemma 7.17]{MR1707339}, that it
generalizes in view of its applications  for Theorem 
\ref{prop:clusters_rentrant-sortant_infinis} and Lemma \ref{lem:connections}.  
\begin{lemma}\label{l1} 
If $\lambda_c<\lambda'$ and $\eta>0$, then there exist integers 
$m=m(\lambda',\eta)$ and $n=n(\lambda',\eta)$
 satisfying \eqref{mnN} and such that 
\begin{eqnarray*}
&&P\big(\mbox{there exists a }\lambda' \mbox{-open path  in }B(n)\cup T^{i,j}(m,n)\mbox{ from }B(m)\\
&&\mbox{to a seed  contained in } T^{i,j}(m,n)\big)>1-\eta,
\end{eqnarray*}
for any choice of $(i,j)\in\{+,-\}^2$.
\end{lemma}
 \begin{notation}\label{sigma-fields} 
Given a subset  $V$ of $\Z^3$,  $\delta>0$ and $x\in V$,  
we let $\sigma(x,V,\lambda',\delta)$ be the $\sigma$-algebra 
generated by the indicator functions of
the following collection of events:
\begin{equation}\label{eq:coll-evt}
\{T_y>e_{\lambda'}(y,z): y\in V, z\sim y \}\cup 
\{T_y>e_{\lambda'+\delta}(y,z):y\in V\cap B_x(n)^c,z\sim y  \}.
\end{equation}
Note that when $V\cap B_x(n)^c=\emptyset$, $\sigma(x,V,\lambda',\delta)$ 
is simply the $\sigma$-algebra generated by the indicator functions of 
$\{T_y>e_{\lambda'}(y,z): y\in V, z\sim y \}$,
which we will denote by  $\sigma(V,\lambda')$.  
\end{notation} 
 For $A$ a subset of $\Z^3$, recall from  \eqref{exteriorvertexboundary} that 
$\Delta_v A$ denotes the exterior vertex boundary of $A$.  
\begin{lemma}\label{l2}
 If $\lambda_c< \lambda'$ and $\epsilon,\delta >0$, there exists $m=m(\lambda',\epsilon,\delta)$
and $n=n(\lambda',\epsilon,\delta)$  satisfying \eqref{mnN} 
 and with the following property: 
\newline For any  choice 
of $(i,j)\in\{+,-\}^2$, any  $x\in \Z^3$,  any set $L\subset\Z^3$ 
such that 
\begin{equation}\label{eq:Lsuchthat}
B_x(m)\subset L\subset\Z^3\setminus T^{i,j}_x(m,n) 
\end{equation}
and for any $\sigma(x,L,\lambda',\delta)$-measurable  
  event $H$ of strictly positive probability we have:
\begin{equation}\label{eq:cl-l2}
P(G^{i,j}\vert H)\geq 1-\epsilon ,
\end{equation}
where
\begin{eqnarray*}
G^{i,j}&=&\big\{\mbox{there exists a path contained in } B_x(n)\cup T^{i,j}_x(m,n) \mbox{ going from  }L\\
&& \mbox{ to a seed contained in  }  T^{i,j}_x(m,n)
\mbox{ and  such that}\\ 
&&\mbox{its first edge }  (u,v) \mbox{ with }u\in L,v\in\Delta_v L  
\mbox{ is }(\lambda'+\delta) \mbox{-open} \\
&&
\mbox{and all}\mbox{ its other edges are } \lambda'\mbox{-open}\big\}.
\end{eqnarray*}
\end{lemma}
Note that since a $\lambda'$-open edge is also $(\lambda'+\delta)$-open, all the involved edges
 in the path  in  $G^{i,j}$ are $(\lambda'+\delta)$-open,  but 
 we do not know if the first edge of this path is
$\lambda'$-open. 
\begin{figure}[htp]\label{fig:dessin_lemme1.5}
\centering
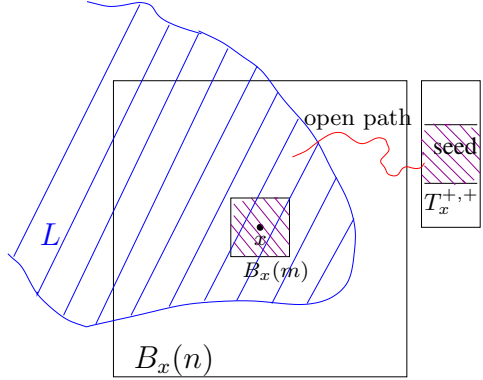
\caption{In the open path, the first edge going out of $L$ is $(\lambda'+\delta)$-open,
and all other edges are $\lambda'$-open.}
\end{figure} 
\begin{proof}{Lemma}{l2} 
 Given $A$ a subset of $\Z^3$ and a subset $C$ of $\Delta_v A$, let 
\begin{equation}\label{def:AbuildrelC}
\{A 
\buildrel{\lambda'+\delta}\over{\Rightarrow}C\}
\end{equation}
 be the event that at least one of the bonds going from $A$ to $C$ is $(\lambda'+\delta)$-open.
Note that this is a stronger condition than to have a $(\lambda'+\delta)$-open path 
from $A$ to $C$. \par
Let $\alpha$ be the probability that any given bond is $\lambda'$-open.\par
Since the model is invariant under translations and 90 degree rotations, 
it suffices to show the lemma when $x$ is the origin and  $(i,j)=(+,+)$.  
We will  hence  drop $x,i$ and $j$ from the notation.
Let 
\begin{eqnarray}
&&V(L)=\big\{ z\in \Delta_v(L\cap B(n)): 
\mbox{ there exists a }\lambda'\mbox{-open path contained in }\nonumber\\
&&B(n)\cup T(m,n)\setminus L\mbox{ going from }z \mbox{ to a seed contained in }T(m,n)\big\}. 
\label{def:V(L)}
 \end{eqnarray}
Here it is understood that there always is a $\lambda'$-open path from a point to itself. 
Therefore any $ z\in \Delta_v(L\cap B(n)) $ contained in a seed in $T(m,n)$ is in $V(L)$. Note also 
that since the path is contained in  $B(n)\cup T(m,n)$,  $V(L)$ is a subset of $B(n)\cup T(m,n)$.
Now write 
\begin{equation}\label{def-G}
G=\cup_{K}\big(\{L 
\buildrel{\lambda'+\delta}\over{\Rightarrow}K\}\cap\{V(L)=K\}\big)
\end{equation}
where the union is over all possible values of $V(L)$.\par
Our next step is to show that if $m$ and $n$ are properly chosen, the set $V(L)$ is 
large with probability close to $1$.
Let $k$ be a positive integer.  Note that  if $V(L)$ has at most $k$ elements, 
by the FKG inequality (see Remark \ref{rk:FKG}), the probability that all the bonds entering 
$V(L)$ are $\lambda'$-closed is at least $(1-\alpha)^{6k}$  (recall that $\alpha$ is 
the probability that any given bond is $\lambda'$-open, and that we are working in 
$\Z^3$). All the $\lambda'$-open paths  going from $L$ to a seed contained in  
$T(m,n)$ have to pass through a point in $V(L)$. Hence if all the bonds entering 
$V(L)$ are $\lambda'$-closed, such a path does not exist. Thus we have
\begin{eqnarray}
&& P\big(\mbox{there exists a }\lambda'\mbox{-open path contained in }B(n)\cup T(m,n)\setminus L
\nonumber\\
&&\quad\mbox{ going from } L \mbox{ to a seed contained in }T(m,n)\,\vert\, 
\,\vert V(L)\vert\leq k\big)\nonumber\\
&&\leq 1-(1-\alpha)^{6k}.\label{(1)}
 \end{eqnarray}
But according to Lemma \ref{l1} there exist $m$ and $n$ such that $2m<n$ and
\begin{eqnarray*}
&& P\big(\mbox{there exists a }\lambda'\mbox{-open path contained in } B(n)\cup T(m,n) \\ 
&&\quad\mbox{ going from }B(m) \mbox{ to a seed contained in }T(m,n)\big) 
 \end{eqnarray*}
is as close to   $1$ as we wish.
 This implies that 
\begin{eqnarray}
&& P\big(\mbox{there exists a }\lambda'\mbox{-open path contained (except its initial point) in}\cr
&&\quad B(n)\cup T(m,n)\setminus L\mbox{ going from }L\cr&&\quad
\mbox{ to a seed contained in }T(m,n)\big) 
\label{csq-lem-1.1}
 \end{eqnarray}
 is as close to $1$ as we wish  uniformly in  $L$'s such that  
\begin{equation}\label{such-sets-L}
B(m)\subset L \subset \Z^3\setminus T(m,n).
\end{equation}
 The probability \eqref{csq-lem-1.1} is equal to 
\begin{eqnarray}
&&
P\big(\mbox{there exists a }\lambda'\mbox{-open path contained (except its initial point) in }\nonumber\\
&&\quad B(n)\cup T(m,n)\setminus L \mbox{ going from }L\mbox{ to a seed contained in }\nonumber\\
&&\quad T(m,n)\,\vert\, \vert V(L)\vert \leq k) P(\vert V(L)\vert \leq k\big)\nonumber\\
&&+
P\big(\mbox{there exists a }\lambda'\mbox{-open path contained (except its initial point) in }\nonumber\\
&&\quad B(n)\cup T(m,n)\setminus L \mbox{ going from }L\mbox{ to a seed contained in }\nonumber\\
&&\quad  T(m,n)\, \vert\, \vert V(L)\vert > k\big)
P( \vert V(L)\vert > k)\nonumber\\
&&\leq
\big(1-(1-\alpha)^{6k}\big) P( \vert V(L)\vert \leq k)+P(\vert V(L)\vert > k)\nonumber\\
&&=1-(1-\alpha)^{6k} P( \vert V(L)\vert \leq k),\label{detail(2)}
 \end{eqnarray}
where the inequality comes from \eqref{(1)}. For this upper bound to be close to 1, 
$P( \vert V(L)\vert \leq k)$ has to be small.  
 Hence, it follows from \eqref{(1)}--\eqref{detail(2)} that 
 for any $\epsilon_0>0$ and any $k\in \N$,
  we can choose $m$ and $n$ with $2m<n$ in such a way that 
\begin{equation}\label{(2)}
P(\vert V(L)\vert \leq k)\leq \epsilon_0
\end{equation}
  for all sets $L$  satisfying \eqref{such-sets-L}. \par 
  Let $K$ be a subset of $ \Delta_v(L\cap B(n)) $. 
  We will now provide a lower bound to $P(L
  \buildrel{\lambda'+\delta}\over{ \Rightarrow}K\, \vert\, H)$ 
  which depends on the cardinality of $K$ but is independent of $H$. 
  Suppose $K$ has at least $6r$ elements.
   Each point $u\in K$ has a neighbor $v\in  L\cap B(n) $, that we associate to $u$.
  But since each point of $\Z^3$ has 6 nearest neighbors,
  $v$ could be a neighbor of up to 6 points of $K$, to which it could have been associated. 
 Then, there exist distinct $x_1,\dots,x_r \in  L\cap B(n) $ and distinct 
 $y_1,\dots,y_r \in  K $ 
 such that $x_i\sim y_i$ for $i=1,\dots,r$. \par
 Since $x_i\in L\cap B(n)$ and $H$ is $\sigma(o,L,\lambda',\delta)$-measurable,  
by Lemma \ref{lem:sprinkling}, for some $\iota >0$ we have
\begin{equation}\label{(3)}
P(X_{\lambda'+\delta}(x_i,y_i)=0\,\vert\, X_{\lambda'+\delta}(x_j,y_j)=0,j=1,\dots i-1;H)<1-\iota 
\end{equation}
for $i=1,\dots,r$. 
It now follows from \eqref{(3)} and an inductive argument that
\[ 
P(X_{\lambda'+\delta}(x_i,y_i)=0,\, i=1,\dots,r\,\vert\, H)<(1-\iota)^r.
\] 
  Hence for all $K\subset \Delta_v(L\cap B(n))$ such that $\vert K\vert \geq 6r$ we have: 
\begin{equation}\label{(4)}
P(L
\buildrel{\lambda'+\delta}\over{ \Rightarrow}K \,\vert\, H)\geq 1-(1-\iota)^r.
 \end{equation}
 Now write:
\begin{eqnarray*}
P(G, V(L)=K\,\vert\, H)&=&P(L
\buildrel{\lambda'+\delta}\over{ \Rightarrow}K,V(L)=K\,\vert\, H)\\
&=&P(L
\buildrel{\lambda'+\delta}\over{ \Rightarrow}K\,\vert\, V(L)=K, H)P(V(L)=K\,\vert\, H). 
 \end{eqnarray*}
 But the event $\{V(L)=K\}$ is measurable with respect to the $\sigma$-algebra 
 generated by the random variables  $(T_x,e_{\lambda'}(x,y):x\notin L)$  while both 
 $\{L
 \buildrel{\lambda'+\delta}\over{ \Rightarrow}K \}$ and $H$ are measurable with 
 respect to the $\sigma$-algebra generated by the random variables
  $(T_x,e_{\lambda'}(x,y),e_{\lambda'+\delta}(x,y):x\in L)$.  
 Therefore  $\{V(L)=K\}$ is independent of the pair of events 
 $H,\{L
 \buildrel{\lambda'+\delta}\over{ \Rightarrow}K \}$, so that
\[
P(G, V(L)=K\,\vert\, H)=P(L
\buildrel{\lambda'+\delta}\over{ \Rightarrow}K\,\vert\,  H)P(V(L)=K)
\]
   Then summing up over all sets $K$ such that $\vert K\vert\geq 6r$, 
   it follows from  \eqref{def-G} and \eqref{(4)} that
\[
P(G\vert H)\geq (1-(1-\iota)^r) P(\vert V(L)\vert\geq 6r).
\]
  To complete the proof of Lemma \ref{l2} first pick $r$ such that
  $(1-\iota)^r<{\epsilon}/{2}$
and then use \eqref{(2)} to pick $m$ and $n$ such that 
$P(\vert V(L)\vert\geq 6r)\geq 1-{\epsilon}/{2}$.
\end{proof}
\mbox{}\\ 

\begin{notation}\label{def:barTx}
 For a given 
$x\in\Z^3$  and $i=1,2,3$, $H_x^i$ will denote the hyperplane perpendicular to 
 ${\rm e}_i$  passing through $x$.
Before stating the next lemma where $x\in B(2N-m)$, 
we define $\overline T_x(m,n)$ to be the thickened box
built from the quadrant opposite to the one $x$ belongs to in the face
$H_x^1\cap B(2N-m)$, that is
\begin{eqnarray*}
T^{+,+}_x(m,n) &\mbox{ if }& x_2\leq 0 \mbox { and }x_3\leq 0,\\
T^{+,-}_x(m,n) &\mbox{ if }& x_2\leq 0 \mbox { and }x_3> 0,\\
T^{-,+}_x(m,n) &\mbox{ if }& x_2> 0 \mbox { and }x_3\leq 0,\\
T^{-,-}_x(m,n) &\mbox{ if }& x_2> 0 \mbox { and }x_3> 0.
\end{eqnarray*}
\end{notation}
Thanks to  Lemma \ref{l2},  in the following lemma we construct
open paths starting 
in a seed inside $B(2N)$, and reaching  either  a seed 
inside $B_{6N{\rm e}_1}(2N)\cap B_{8N{\rm e}_1}(2N)$ or
 inside $B_{6N{\rm e}_2}(2N)\cap B_{8N{\rm e}_2}(2N)$.
For $i\in\{1,2\}$, the successive seeds in these open paths will have centers belonging  to the
 hyperplanes $H^i_{x+N{\rm e}_i},  H^i_{x+2N{\rm e}_i},
 H^i_{x+3N{\rm e}_i}$, $H^i_{x+4N{\rm e}_i},\ldots$; 
 these successive seeds will be respectively contained in  
 $B_{N{\rm e}_i}(2N)$, $B_{2N{\rm e}_i}(2N),
 B_{3N{\rm e}_i}(2N), B_{4N{\rm e}_i}(2N),\ldots$, 
 and we 
 will stop as soon as we will get a seed in $B_{8N{\rm e}_i}(2N)$.
This construction will use a \textit{steering procedure} in which, at each stage,
the choice of a seed in $\overline T_.(m,n)$ compensates from an earlier deviation. 
\begin{lemma}\label{l3} Given $\lambda'>\lambda_c$, for any $\epsilon,\delta >0$ 
there exist  $n=n(\lambda',\epsilon,\delta),m=m(\lambda',\epsilon,\delta)$ and $N$  
satisfying \eqref{mnN},  such that for any $x\in B(2N-m)$,
\begin{eqnarray*}
P(C_x^i) \geq 1-8\epsilon, \qquad \mbox{ for }  i\in\{1,2\}
\end{eqnarray*}
where 
\begin{eqnarray*}
C_x^i&=&\big\{ \mbox{there exists a seed }B_{y}(m) \mbox{ contained in }B_{8N{\rm e}_i}(2N),
\mbox{ with }y_i\le 8N \\
&&
\mbox{ and a path contained in }B(3N)  \cup B_{6N{\rm e}_i}(3N)\\
&&\mbox{ from }
B_x(m) \mbox{ to } B_{y}(m) \mbox { whose edges are all }(\lambda'+\delta)\mbox{-open and }\\
&&\mbox{those which are to the right of (resp. above) 
 the hyperplane }\\
&& H^i_{y-N{\rm e}_i} \mbox{ when $i=1$ (resp. $i=2$) are }\lambda'\mbox{-open}\big\}
\end{eqnarray*} 
\end{lemma}
\begin{figure}[htp]\label{fig:dessin_lemme1.7}
\centering
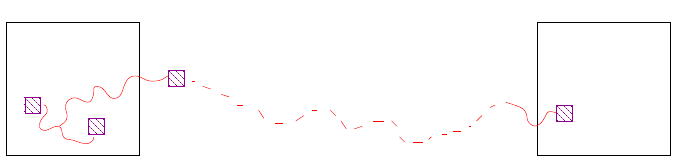
\caption{Event $C_x^1$}
\end{figure}
\begin{proof}{Lemma}{l3}
We consider $C_x^1$. Since the model is invariant under $90$ degree rotations, 
the proof will also be valid for $C_x^2$.
 Let $V_1$ be the set of vertices of all paths starting at $B_x(m)$ 
 and  contained in $B_x(n)\cup \overline T_x(m,n)$
whose first edge is $(\lambda'+\delta)$-open and all the other edges are $\lambda'$-open.  
 Let
$$A^1_x=\{V_1\mbox{ contains a seed in } \overline T_x(m,n)\}.$$
Note that since $x\in B(2N-m)$  a path contained in $B_x(n)\cup \overline T_x(m,n)$ 
is also contained in $B(3N)$. Note also that 
the center of a seed in $\overline T_x(m,n) $ belongs to
 $H^1_{x+N{\rm e}_1}$, and that by our definition of $\overline T_x(m,n)$ (in Notation \ref{def:barTx}) 
 this seed is contained in $B_{N{\rm e}_1}(2N)$.
Thanks to  Lemma \ref{l2} with $L=B_x(m)$ and
$H$ the whole probability space (that is without conditioning), 
there exist $n,m$ such that 
\begin{equation}\label{PA1}
P(A^1_x)>1-\epsilon.
\end{equation}
 If $A^1_x$ occurs, of all the seeds in $\overline T_x(m,n)\cap V_1$ 
 we choose one according to some arbitrary deterministic order. 
 We now define a random variable  $Z_1$ as follows: on $A^1_x$, $Z_1$ 
 is the center of the chosen seed and on $(A^1_x)^c$,
 $Z_1=\Delta$  where $\Delta$ is an extra point we add to $\Z^3$. 
 Note that on $A^1_x$, $Z_1$ takes values in the hyperplane 
 $H^1_{x+N{\rm e}_1}$.  The random variable $Z_1$ is a function of $V_1$ 
 which we denote by $F_1$. We now wish to give a lower bound to the 
 conditional probability given $\{Z_1=z_1\}$ with 
 $z_1\neq \Delta $  that there is a path contained in
  $B_{z_1}(n)\cup \overline T_{z_1}(m,n)$ from 
 $V_1$ to a seed in $\overline T_{z_1}(m,n)$ and having the following properties:\par
\textit{(i)} all its bonds are $(\lambda'+\delta)$-open; \par
\textit{(ii)} all its bonds 
which are to the right of $H^1_{x+(N+m){\rm e}_1}$ are $\lambda'$-open.
\par\noindent
 Therefore to obtain the lower bound we let $L_1$ be a value of $V_1$ containing  
 a seed in $\overline T_x(m,n)$ and  consider the event:
\begin{eqnarray*}
A^1(x,L_1)&=&\big\{\mbox{there exist }v_1\in L_1\cap B_{F_1(L_1)}(n) \mbox{ and 
 a path}\\
&&  \mbox{from   }v_1 \mbox{ to a seed in } \overline T_{F_1(L_1)}(m,n),
  \mbox{contained in } \\
&&  B_{F_1(L_1)}(n)\cup    \overline T_{F_1(L_1)}(m,n) \mbox{  whose edges are  all }(\lambda'+\delta) \mbox{-open}\\
&&   \mbox{and those to the right of }
H^ 1_{x+(N+m){\rm e}_1} \mbox{ are } \lambda'\mbox{-open}\big\}.
\end{eqnarray*}
The event $\{V_1=L_1\}$ is  $\sigma(L_1,\lambda')$-measurable (recall Notation \ref{sigma-fields}),
hence it follows from Lemma \ref{l2} that
\begin{equation}\label{PA1cond}
P(A^1(x,L_1)\vert V_1=L_1)\geq 1-\epsilon.
\end{equation}
Let $V_2$ be
the set of vertices of all the paths 
 with the following properties: \par
\textit{(i)} they start from $B_x(m)$; \par
\textit{(ii)} they are contained in $B(3N)\cup B_{N{\rm e}_1}(3N)$ and 
lie entirely to the left of $H^1_{x+(2N+m){\rm e}_1}$; \par
\textit{(iii)} all their edges are $(\lambda'+\delta)$-open and  those 
to the right of $H^ 1_{x+(N+m){\rm e}_1}$ are $ \lambda' $-open. \par
We also define the event
$$A^2_x=\{V_2 \mbox{ contains a seed centered in }  H^1_{x+2N{\rm e}_1} \cap B_{2N{\rm e}_1}(2N-m)\}.$$
Noting that $A^2_x$ contains $A^1(x,L_1) \cap \{V_1=L_1\}$  for any $L_1$ containing a seed 
in $\overline T_x(m,n)$, summing over all such $L_1$'s we get by \eqref{PA1} and \eqref{PA1cond} 
\begin{eqnarray} 
P(A^2_x)&\geq& \sum_{L_1}  P(A^1(x,L_1)\vert V_1=L_1)P(V_1=L_1)\cr 
&\geq& \sum_{L_1} (1-\epsilon)P(V_1=L_1)=(1-\epsilon)P(A^1_x)\cr 
\label{eq:PA2-l3}
&\geq&(1-\epsilon)^2 \geq 1-2\epsilon.
 \end{eqnarray} 
Now we define a random variable $Z_2$ as follows: on the event $A^2_x$ among the seeds contained in 
$V_2$ and centered in  $H^1_{x+2N{\rm e}_1} \cap B_{2N{\rm e}_1}(2N-m)$
we choose one according to some arbitrary deterministic order and we let $Z_2 $ be its center. 
On  $(A^2_x)^c$  we let
$Z_2=\Delta $. Thus, $Z_2$ is a function $F_2$ of $V_2$. As before we let $L_2$ be a possible value of $V_2$ containing a seed 
centered in  $H^1_{x+2N{\rm e}_1} \cap B_{2N{\rm e}_1}(2N-m)$ and consider the event: 
\begin{eqnarray*}
A^2(x,L_2)&=&\big\{ \mbox{there exist }v_2\in L_2\cap B_{F_2(L_2)}(n) 
\mbox{ and a path from }v_2 \mbox{ to a} \\
&&\mbox{seed in } \overline T_{F_2(L_2)}(m,n),\mbox{ contained in }   B_{F_2(L_2)}(n)\cup    \overline T_{F_2(L_2)}(m,n)\\
&& \mbox{whose edges are  all }(\lambda'+\delta) \mbox{-open  and those to the right of }\\
&&H^ 1_{x+(2N+m){\rm e}_1} \mbox{ are } \lambda'\mbox{-open}\big\}. 
\end{eqnarray*}
The event  $\{V_2=L_2\}$ is $\sigma(F_2(L_2),L_2,\lambda',\delta)$-measurable, hence 
  it follows from Lemma \ref{l2} that
  $$P(A^2(x,L_2)\vert V_2=L_2)\geq 1-\epsilon. $$
We now let $V_3$ be the set of vertices belonging to all the paths with the following properties: \par
\textit{(i)} they start from $B_x(m)$; \par
\textit{(ii)} they are contained in $B(3N)\cup B_{N{\rm e}_1}(3N)$ and lie entirely to the left of $H^1_{x+(3N+m){\rm e}_1}$; \par
\textit{(iii)} all their edges are $(\lambda'+\delta)$-open and  those to the right of $H^ 1_{x+(2N+m){\rm e}_1}$ are $ \lambda' $-open. \par
 We also define the event:
$$A^3_x=\{V_3 \mbox{ contains a seed centered in }  H^1_{x+3N{\rm e}_1} \cap B_{3N{\rm e}_1}(2N-m)\}.$$
  Since $ A^3_x$ contains $A^2(x,L_2) \cap \{V_2=L_2\}$ we can argue  as before and get:
   $$P(A^3_x)\geq 1-3\epsilon .$$
The argument is then repeated until we reach a seed in $B_{8N{\rm e}_1}(2N)$. 
The total number of steps needed is at most $8$. 
Since at each step the probability is reduced by $\epsilon$, the lemma is proved.
\end{proof} 
\mbox{}\\ \\
Then define
\begin{equation}\label{eq:Cx}
C_x=C^1_x\cap C^2_x.
\end{equation}
{}From Lemma \ref{l3} we get:
\begin{corollary}\label{c1} 
Given $\lambda'>\lambda_c$, for any $\epsilon,\delta >0$ there exist $n,m,N$ 
satisfying \eqref{mnN} and such that for any 
$x\in B(2N-m)$ we have
$$P(C_x)\geq 1-16\epsilon.$$
\end{corollary}
 Next lemma fixes  the values of all the parameters introduced up to now.
\begin{lemma}\label{lsuppl} 
Assume $\lambda >\lambda_c$. Then, there exist constants
$m, N , K$ and $\iota>0$ such that for all $k$, 
\begin{eqnarray*}
&&P(\mbox{there exists a }\lambda\mbox{-open  path contained in }\\
&&[-3N,(3+8k)N]\times [-3N,(3+8k)N]\times [-3N,3N] \\
&&\mbox{from }B(m)\mbox{ to a seed in }  B_{8Nk{\rm e}_1+8Nk{\rm e}_2}(2N)\\
&&\mbox{whose number of edges is at most }2Kk) \geq \iota.
\end{eqnarray*} 
\end{lemma}
\begin{proof}{Lemma}{lsuppl}
We first fix $\epsilon>0$ small enough for the two dimensional oriented site percolation of parameter
 $1-16\epsilon$ to be supercritical. Then we take
$\lambda'> \lambda_c$ and $\delta >0$ such that $\lambda'+\delta<\lambda$. Finally
for those values of $\epsilon$, $\delta$ and $\lambda'$ we fix $n$, $m$ and  $N=n+m+1$
satisfying \eqref{mnN} and such that the conclusion of Corollary \ref{c1} is valid. \\ \\ 
We create a two dimensional oriented site percolation on $(\Z_+)^2$ associated to the percolation model
we already have. We will refer to this model as the ``renormalized model'', 
while the percolation model we already had on $\Z^3$  will be referred to 
as the ``original model''. On the renormalized model
all the paths are oriented upwards and towards the right; moreover, two subsequent sites 
of a path are at euclidean distance $1$. We now explain the way in which these models 
are associated. In the renormalized model 
site $(0,0)$ is always considered open, site $(0,1)$  is open (closed) if $C^1_0$ occurs 
(does not occur) in the original model. Similarly, $(1,0)$ is open (closed) if $C^2_0$ occurs 
(does not occur) in the original model.
Note that although the states of these last two sites $(0,1)$ and $(1,0)$ are dependent, 
by Corollary \ref{c1} they are both open with probability at least $1-16\epsilon$. 
We then proceed recursively 
as follows.\\ \\ 
At the $n$-th step we will look at the points in $\{(x,y)\in \Z^2_+: x+y=n-1\}$
which have been reached in the renormalized model from
 $(0,0)$ following open paths and order them according to their second coordinates. 
 We start from the point having the lowest second coordinate.
 Assume it is $(x_1, n-1-x_1)$. This point was reached from either 
 $(x_1-1, n-1-x_1)$ or $(x_1, n-2-x_1)$. In the first case, in the original model 
 a seed is reached in the left portion of  $B_{8Nx_1{\rm e}_1+8N(n-1-x_1){\rm e}_2}(2N)$  
(remember the description given before the statement of Lemma \ref{l3}). 
 Let  $z_1$ be the center of this seed. If  $C^1_{z_1}$ occurs 
 (does not occur) in the original model  we say that site $(x_1+1, n-1-x_1)$ in the renormalized model 
 is open (closed). And if  $C^2_{z_1}$ occurs 
(does not occur) in the original model we say that site $(x_1, n-x_1)$ is open (closed). Note that since 
 $z_1$ is in the left portion of $B_{8Nx_1{\rm e}_1+8N(n-1-x_1){\rm e}_2}(2N)$,  when we  
attempt to move upwards, the first seed we are seeking is centered  to the right of  $z_1$ due 
to our steering procedure, thus avoiding regions where we have already used $(\lambda'+\delta)$-open 
edges.
In the second case, the seed  reached in the original model (we again denote its center by $z_1$)
is in the lowest portion of $B_{8Nx_1{\rm e}_1+8N(n-1-x_1){\rm e}_2}(2N)$  and when we want 
 to establish if $C^1_{z_1}$ occurs we will be looking for paths reaching a seed whose center 
 is above $z_1$.  
  We then move to the second point in 
 $\{(x,y)\in \Z^2_+: x+y=n-1\}$ which has been reached in the renormalized model from
 $(0,0)$ following open paths. Let $(x_2,n-1-x_2)$ be that point and let $z_2$ be the 
 center of the seed located inside $B_{8Nx_2{\rm e}_1+8N(n-1-x_2){\rm e}_2}(2N)$ 
 which was reached in the original model following open paths starting at $B(m)$. 
 Two different cases arise: either $x_2=x_1-1$ or $x_2<x_1-1$. In the first case the point 
$(x_2+1,n-1-x_2)=(x_1,n-x_1)$ 
 has already been declared open or closed and remains in that state. 
 Then, we declare $(x_2,n-x_2)$ open  (closed)  if $C^2_{z_2}$ occurs (does not occur) in the original model.  
 In the second case (when $x_2<x_1-1$) we declare $(x_2+1,n-1-x_2)$ open (closed) if  
 $C^1_{z_2}$  occurs (does not occur) in the original model  and we declare $(x_2,n-x_2)$ 
 open (closed) if  $C^2_{z_2}$ occurs (does not occur) in the original model. 
 Then we go on.\\ \\ 
We now note that for all $n$ each  site examined in the set  
$\{(x,y):x+y=n\}$ has probability bigger than $1-8\epsilon$ 
of being open and that such sites are dependent at most by pairs.  
This implies, as explained in the following lines,  that 
the open cluster of the origin is stochastically above the open cluster 
of an independent oriented site percolation model of parameter $1-16\epsilon$. \\ \\
For this, we again proceed by induction on $n$. We denote by $a_1,a_2,\ldots,a_k$ the points in the open
cluster of the origin that belong to $\{(x,y)\in \Z^2_+: x+y=n\}$. We assume that they are ordered according 
to their second coordinates.
Point $a_1$ has two neighbors $b_1,b_2$ on $\{(x,y)\in \Z^2_+: x+y=n+1\}$.
They are both open with probability at least $1-16 \epsilon$, which is stochastically larger
than if they were both independently open with probability $1-16 \epsilon$.
In other words, if a random vector $(Y_1,Y_2)$ with coordinates taking values in $\{0,1\}$
is such that $ P(Y_1=Y_2=1)\geq 1-16 \epsilon$, then the vector $(Y_1,Y_2)$ is stochastically larger
than the vector $(X_1,X_2)$ where $X_1$ and $X_2$ are independent Bernoulli r.v.'s of
parameter $1-16 \epsilon$. Going on, if
$a_2 =a_1+(-1,1)$, then we just have to consider the point $b_3=a_2+(0,1)$, because $a_2+(1,0)$ has already been
examined. This point $b_3$ will be open with probability at least $1-8\epsilon$ independently of what
happened with $b_1$ and $b_2$. Otherwise if $a_2$ is more distant from $a_1$ we have to examine $b_3= a_2+(1,0)$
and $b_4=a_2+(1,0)$: they will both be open with probability at least $1-16 \epsilon$ independently of what
happened with $b_1$ and $b_2$, and so on.
In the end, to each examined point on $\{(x,y)\in \Z^2_+: x+y=n+1\}$ 
is attached a r.v. with value 1 if it is open and 0 if it is closed. 
The r.v.'s thus obtained are stochastically larger
than a sequence of independent Bernoulli r.v.'s of
parameter $ 1-16 \epsilon$.
\\ \\ 
Thus, for  our choice of $\epsilon$ 
the renormalized model is supercritical and there exists a constant 
 $\iota >0$  such that $P((0,0)\rightarrow (k,k))\geq \iota$ for all $k\in \N$. 
Note also that the existence of an open oriented path from $(0,0)$ 
to $(k,k)$  (which has length $2k$) in the renormalized model  
implies the existence of a $(\lambda'+\delta)$-open path in the original model from $B(m)$ to some seed
in $B_{8Nk{\rm e}_1+8Nk{\rm e}_2}(2N)$ 
whose number of edges is bounded above by $2Kk$ where $K$ is some constant 
that depends on $N$ but not on $k$.  Indeed suppose that the point following $(0,0)$ in the path of
 the renormalized model is $(1,0)$. This means that there exists an open path in the original model 
 from a seed in $B(2N)$ to a seed in
$B_{8Ne_1}(2N)$.
This last path is not oriented, but being contained in
$B(3N) \cup B_{6N{\rm e}_1}(3N)$, it uses only edges in this set. The total number of
edges in the latter is a function of $N$ which does not depend on $k$, that we denote 
by $K(N)$. Hence the derived open path in the original model from a seed in $B(2N)$ 
to a seed in  $B_{8Nk{\rm e}_1+8Nk{\rm e}_2}(2N)$ has a  number of edges bounded by $2K(N)k$. 
\end{proof} 
\mbox{}\\ \\
For our next result we define the boxes:
$$\overline{B}_{i,j}=B_{(3+8i)N{\rm e}_1+(3+8j)N{\rm e}_2}(2N)$$
where $i$ and $j$ are non-negative integers.
\begin{corollary}\label{C1}
 Assume $\lambda >\lambda_c$. 
Let $N$ be as in the conclusion of Lemma \ref{lsuppl}. 
Then, there exist $\iota'>0$ and $K'\in \N$ such that: 
For any $k\in \N$ and any $0\leq i_1,i_2,j_1,j_2\leq k$  we have 
\begin{eqnarray*}
&&P(\mbox{there exists a }\lambda\mbox{-open  path contained in }\\
&&[0,(6+8k)N]\times [0,(6+8k)N]\times [-3N,3N] \mbox{ from }\overline{B}_{i_1,i_2}\mbox{ to  } \\
&& \overline{B}_{j_1,j_2} \mbox{  whose number of edges is at most } 2K'(\vert i_1-j_1\vert
 +\vert i_2-j_2\vert)) \geq \iota'.
\end{eqnarray*} 
\end{corollary}
\begin{proof}{Corollary}{C1} 
We wish to join  $\overline{B}_{i_1,i_2}$ to $\overline{B}_{j_1,j_2}$.  
Lemma \ref{lsuppl} enables to go from a box to another one along a  diagonal 
direction  issued from that box. Hence applying Lemma \ref{lsuppl} we get
\begin{eqnarray*}
&&P(\mbox{there exists a }\lambda\mbox{-open  path contained in }\\
&&[0,(6+8k)N]\times [0,(6+8k)N]\times [-3N,3N] \mbox{ from }\overline{B}_{i_1,i_2} \mbox{ to  } \\
&& \overline{B}_{i_1+r,i_2+r} \mbox{  whose number of edges is at most }   2K 
r) \geq \iota 
\end{eqnarray*}
for all $r \in \N$ such that $i_1+r, i_2+r \leq k$. Since the percolation model 
is invariant under 90 degree rotations, the same inequality holds if instead of  
adding $(r,r)$ to $(i_1,i_2)$ we add $(r,-r)$,$(-r,r)$ or $(-r,-r)$.  That is, instead of
going in one direction of one diagonal issued from $(i_1,i_2)$, we may take this diagonal 
in the other direction, or one direction of the other diagonal issued from $(i_1,i_2)$. 
This depends on the relative positions of $(i_1,i_2)$ and $(j_1,j_2)$ within the square
$[0,k]\times[0,k]$ to which they both belong. More precisely, from  $(i_1,i_2)$ and from $(j_1,j_2)$
 is issued a diagonal, and those two diagonals intersect within $[0,k]\times[0,k]$.
If this intersection point has integer coordinates, it can be written
$(i_1+r_1,i_2+r_2)$ as well as  $(j_1+\ell_1,j_2+\ell_2)$, with 
 $r_1=r_2$ or $r_1=-r_2$ (depending on which diagonal issued from
$(i_1,i_2)$ was used), and with  $\ell_1=\ell_2$ or $\ell_1=-\ell_2$ similarly.
If this intersection point does not have integer coordinates,
on each of the involved diagonals there is one point with integer coordinates,
with those two points at distance 1,
of the form $(i_1+r_1,i_2+r_2)$ and $(j_1+\ell_1, j_2+\ell_2)$, always with 
$r_1=r_2$ or $r_1=-r_2$, and $\ell_1=\ell_2$ or $\ell_1=-\ell_2$.
To summarize,  there exist integers $r_1,r_2,\ell_1,\ell_2$ with the following properties
\begin{enumerate}
\item $r_2=r_1$ or $r_2=-r_1$ and $\ell_2=\ell_1$ or $\ell_2=-\ell_1$;
\item $0\leq i_1+r_1,i_2+r_2,j_1+\ell_1,j_2+\ell_2\leq k$;
\item either 
$\vert i_1+r_1-(j_1+\ell_1)\vert+\vert i_2+r_2-(j_2+\ell_2)\vert=0$ \newline or
 $\vert i_1+r_1-(j_1+\ell_1)\vert+\vert i_2+r_2-(j_2+\ell_2)\vert=1$;
\item $\vert r_1\vert +\vert \ell_1\vert \leq \vert i_1-j_1\vert +\vert i_2-j_2\vert$.
\end{enumerate}
\begin{figure}[htp]\label{fig:dessin_corollary1.10}
\centering
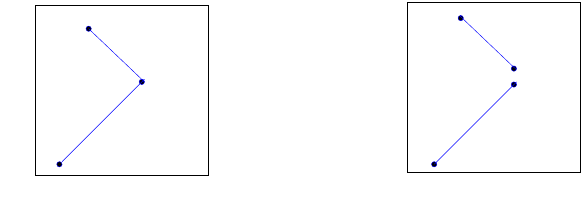
\caption{Two possible cases}
\end{figure}
The corollary now follows from Lemma \ref{lsuppl}, the FKG inequality  (see Remark \ref{rk:FKG}) 
and the fact that the distance from a point in 
$\overline{B}_{i_1+r_1,i_2+r_2}$ to a point in $\overline{B}_{j_1+\ell_1,j_2+\ell_2}$ 
is bounded above by  $20N$. 
\end{proof}
\mbox{}\\ \\
\begin{proposition}\label{p1}
Suppose $\lambda >\lambda_c$. Then there exist constants $C, N$ and $\delta_1 >0$ such that

a) for all $M\geq 6N$,  $x,y\in [0,M]\times [0,M]\times [-3N,3N]$,
\begin{eqnarray*}
&&P\big(\mbox{there exists an open path from } x\mbox{ to }y\mbox{ contained in }\\
&&[0,M]\times [0,M]\times[-3N,3N]\mbox { with at most }C\|x-y\|_1 \mbox{ edges}\big)\geq \delta_1
\end{eqnarray*}
b) The original model is supercritical in a slab on thickness $k=6N$.
\end{proposition}
\begin{proof}{Proposition}{p1}
 It   follows from Lemma \ref{lsuppl} that the probability of having an open path of length $n$ 
 starting in $B(m)$ and contained
in the slab $ \Z \times \Z\times [-3N,3N] $ does not converge to $0$ as $n$ goes to infinity. 
This proves part\textit{ b)}. To prove part \textit{a)} consider the boxes
$B_{(3N+8Ni){\rm e}_1+(3N+8Nj){\rm e}_2}$ with $0\leq i,j\leq (\frac{M}{N}-6)\frac{1}{8}$. 
 Then, note that  for any point in $[0,M]\times [0,M]\times [-3N,3N]$ there is  such 
a box at distance at most $12N$. The result now follows from this, the FKG inequality 
(see Remark \ref{rk:FKG}) and Corollary \ref{C1}.
\end{proof}
\mbox{}\\ \\
We have now all the ingredients for the  proofs of 
Theorem \ref{prop:clusters_rentrant-sortant_infinis}, Lemma \ref{lem:connections},
and \eqref{eq:analogue_G-thm8.21} of Lemma \ref{lem:exp_decay_R}.\\ \\
\noindent
 \textit{Proof of Theorem \ref{prop:clusters_rentrant-sortant_infinis}}. 
Let $x$ and $y$ be two points in a slab of thickness $6N$. 
By Proposition \ref{p1}, the probability to have an open path from $x$ to $y$ in the slab
is larger than $\delta_1$. Therefore the probability for the outgoing cluster from $x$ in the slab to be
 infinite, as well as the probability for the incoming cluster
to $y$ in the slab  to be infinite,  is at least $\delta_1$.

Note that Proposition \ref{p1} gives more precise information, 
since it restricts the involved open paths to a part of the slab, 
and gives an upper bound on the lengths of the paths.
\hfill\mbox{$\square$}
\mbox{}\\ \\
\noindent
 \textit{Proof of Lemma \ref{lem:connections}}. 
For two points $x$ and $y$, the idea to build an open path from $x$ to $y$ is to combine paths in
different slabs using in each one Proposition \ref{p1},\textit{a)}.\\ \\
\textit{(i)} Let $\delta_1,M$ and $C$ be given by Proposition \ref{p1}, and let $k\ge M$. For $n>0$, 
let $x=(x_1,x_2,x_3)\in B_{n+k}\setminus B_n,
\ y=(y_1,y_2,y_3)\in (B_{n+k}\setminus B_n)\cup \Delta_v(B_{n+k}\setminus B_n )$. 
Assume for instance that $x_1<-n$, $n<y_1$, $-n<x_2<n$ and $-n<y_2<n$. 
Let $u,v\in B_{n+k}\setminus B_n$ with $-n<u_1,n<u_2$
and $n<v_1,v_2$.
By Proposition \ref{p1},\textit{a)} there exist with a probability larger than 
$\delta_1$ an open path from $x$ to $u$, as well as  from $u$ to $v$ and
from $v$ to $y$. 
By FKG inequality (see Remark \ref{rk:FKG}) there exists 
therefore with a probability larger than $\delta_1^3$ an open path from $x$ 
 to $y$. Since this particular case gives the maximal distance between $x$ and $y$,
$\delta=\delta_1^3$ enables us to conclude.\\ \\
\textit{(ii)} Let $n<m,\ x \in A(n,m,0),y\in A(n,m,0)\cup \Delta_v A(n,m,0)$. 
We proceed similarly to \textit{(i)}.
Assume for instance that $x_1<n,\,x_2<0$ and $m<y_1,\,y_2<0$. Let $u,v\in A(n,m,0)$ be such that
$u_1<n,\,0<u_2$ and $m<v_1,\,0<v_2$. By Proposition \ref{p1},\textit{a)} 
there exist with a probability larger than $\delta_1$ an open path from $x$ to $u$, 
as well as  from $u$ to $v$ and
from $v$ to $y$. We conclude with $\delta=\delta_1^3$ and $C_1=C$.\\ \\
 Note that  we have to add  $(-x_2)^+ +(-y_2)^+$
in part \textit{(ii)} of the lemma because if 
 $x\in \{z: -k+n\leq z_1<n, -\infty<z_2\leq 0\}$ and 
$y\in \{z:m<z_1\leq m+k, -\infty<z_2\leq 0\}$,  to move from $x$ to $y$ staying 
in  $A(n,m,0)$  we  need  to reach first the set 
$ \{z:-k+n\leq z_1\leq m+k, 0<z_2\leq k\}$ (i.e. to increase the second coordinate 
until it is positive).
\hfill\mbox{$\square$}
\mbox{}\\ \\
\noindent
 \textit{Proof of \eqref{eq:analogue_G-thm8.21} of Lemma \ref{lem:exp_decay_R}}. 
Relying on Proposition \ref{p1},\textit{b)}, we can
follow  the proof of \cite[Theorems (8.18), (8.21)]{MR1707339} to derive 
\eqref{eq:analogue_G-thm8.21}. 
\hfill\mbox{$\square$}
\end{appendix}
\mbox{}\\ \\
\noindent{\bf Acknowledgements.} We thank Geoffrey Grimmett 
for useful discussions.
 We thank referees for helpful comments and suggestions. 
This work was initiated during the semester
``Interacting Particle Systems, Statistical Mechanics and Probability 
Theory'' at CEB, IHP (Paris), whose hospitality is acknowledged. 
Part of this paper was written while  
E.A. was visiting IMPA, Rio de Janeiro and thanks are given for 
the hospitality encountered there. 
\bibliographystyle{alea3}
\bibliography{ACS-for-HAL-V4.bib}

\begin{thebibliography}{19}
\providecommand{\natexlab}[1]{#1}
\providecommand{\url}[1]{\texttt{#1}}
\providecommand{\urlprefix}{URL }
\expandafter\ifx\csname urlstyle\endcsname\relax
  \providecommand{\doi}[1]{doi:\discretionary{}{}{}#1}\else
  \providecommand{\doi}{doi:\discretionary{}{}{}\begingroup
  \urlstyle{rm}\Url}\fi
\providecommand{\eprint}[2][]{\url{#2}}

\bibitem[{Antal and Pisztora(1996)}]{MR1404543}
P.~Antal and A.~Pisztora.
\newblock On the chemical distance for supercritical bernoulli percolation.
\newblock \emph{Ann. Probab.} \textbf{24}~(2), 1036--1048 (1996).
\newblock \href{http://www.ams.org/mathscinet-getitem?mr=MR1404543}{MR1404543}.

\bibitem[{Barsky et~al.(1991)Barsky, Grimmett and Newman}]{MR1124831}
D.~J. Barsky, G.~R. Grimmett and C.~M. Newman.
\newblock Percolation in half-spaces: equality of critical densities and
  continuity of the percolation probability.
\newblock \emph{Probab. Theory Related Fields} \textbf{90}~(1), 111--148
  (1991).
\newblock \href{http://www.ams.org/mathscinet-getitem?mr=MR1124831}{MR1124831}.

\bibitem[{Van~den Berg et~al.(1998)Van~den Berg, Grimmett and
  Schinazi}]{MR1624925}
J.~Van~den Berg, G.~R. Grimmett and R.~B. Schinazi.
\newblock Dependent random graphs and spatial epidemics.
\newblock \emph{Ann. Appl. Probab.} \textbf{8}~(2), 317--336 (1998).
\newblock \href{http://www.ams.org/mathscinet-getitem?mr=MR1624925}{MR1624925}.

\bibitem[{Cerf and Th\'eret(2014)}]{CT}
R.~Cerf and M.~Th\'eret.
\newblock Weak shape theorem in first passage percolation with infinite passage
  times.
\newblock \emph{preprint}  (2014).
\newblock \href{http://arxiv.org/abs/1404.4539}{arxiv.org/abs/1404.4539}.

\bibitem[{Chabot(1998)}]{C}
N.~Chabot.
\newblock \emph{Forme asymptotique pour un mod\`ele \'epid\'emique en dimension
  sup\'erieure \`a trois} (1998).
\newblock Th\`ese de doctorat, Universit\'e de Provence
  \href{http://www.sudoc.fr/011036907}{}.

\bibitem[{Cox and Durrett(1981)}]{MR0624685}
J.~T. Cox and R.~Durrett.
\newblock Some limit theorems for percolation processes with necessary and
  sufficient conditions.
\newblock \emph{Ann. Probab.} \textbf{9}~(4), 583--603 (1981).
\newblock \href{http://www.ams.org/mathscinet-getitem?mr=MR0624685}{MR0624685}.

\bibitem[{Cox and Durrett(1988)}]{MR0978353}
J.~T. Cox and R.~Durrett.
\newblock Limit theorems for the spread of epidemics and forest fires.
\newblock \emph{Stochastic Process. Appl.} \textbf{30}~(2), 171--191 (1988).
\newblock \href{http://www.ams.org/mathscinet-getitem?mr=MR0978353}{MR0978353}.

\bibitem[{Grimmett(1999)}]{MR1707339}
G.~Grimmett.
\newblock \emph{Percolation}.
\newblock Springer-Verlag, Berlin (1999).
\newblock Second edition. Grundlehren der Mathematischen Wissenschaften
  [Fundamental Principles of Mathematical Sciences] Vol. 321
  \href{http://www.ams.org/mathscinet-getitem?mr=MR1707339}{MR1707339}.

\bibitem[{Grimmett and Marstrand(1990)}]{MR1068308}
G.~R. Grimmett and J.~M. Marstrand.
\newblock The supercritical phase of percolation is well behaved.
\newblock \emph{Proc. Roy. Soc. London Ser. A} \textbf{430}~(1879), 439--457
  (1990).
\newblock \href{http://www.ams.org/mathscinet-getitem?mr=MR1068308}{MR1068308}.

\bibitem[{Harris(1960)}]{MR0115221}
T.~E. Harris.
\newblock A lower bound for the critical probability in a certain percolation
  process.
\newblock \emph{Proc. Cambridge Philos. Soc.} \textbf{56}, 13--20 (1960).
\newblock \href{http://www.ams.org/mathscinet-getitem?mr=MR0115221}{MR0115221}.

\bibitem[{Kelly(1977)}]{K}
F.P. Kelly  (1977).
\newblock In discussion of \cite{MR0496851}, pp. 318--319.

\bibitem[{Kesten(1986)}]{MR0876084}
H.~Kesten.
\newblock \emph{Aspects of first passage percolation}.
\newblock Lecture Notes in Math., \textbf{1180}, Springer, Berlin (1986).
\newblock \'Ecole d'\'et\'e de probabilit\'es de Saint-Flour XIV (1984).
  \href{http://www.ams.org/mathscinet-getitem?mr=MR0876084}{MR0876084}.

\bibitem[{Kuulasmaa(1982)}]{MR0675138}
K.~Kuulasmaa.
\newblock The spatial general epidemic and locally dependent random graphs.
\newblock \emph{J. Appl. Probab.} \textbf{19}~(4), 745--758 (1982).
\newblock \href{http://www.ams.org/mathscinet-getitem?mr=MR0675138}{MR0675138}.

\bibitem[{Kuulasmaa and Zachary(1984)}]{MR0766826}
K.~Kuulasmaa and S.~Zachary.
\newblock On spatial general epidemics and bond percolation processes.
\newblock \emph{J. Appl. Probab.} \textbf{21}~(4), 911--914 (1984).
\newblock \href{http://www.ams.org/mathscinet-getitem?mr=MR0766826}{MR0766826}.

\bibitem[{Liggett(2005)}]{MR2108619}
T.M. Liggett.
\newblock \emph{Interacting particle systems}.
\newblock Springer-Verlag, New York (2005).
\newblock Classics in Mathematics (Reprint of first edition).
  \href{http://www.ams.org/mathscinet-getitem?mr=MR2108619}{MR2108619}.

\bibitem[{Mollison(1977)}]{MR0496851}
D.~Mollison.
\newblock Spatial contact models for ecological and epidemic spread.
\newblock \emph{J. Roy. Statist. Soc. Ser. B} \textbf{39}~(3), 283--326 (1977).
\newblock \href{http://www.ams.org/mathscinet-getitem?mr=MR0496851}{MR0496851}.

\bibitem[{Mollison(1978)}]{MR0480208}
D.~Mollison.
\newblock Markovian contact processes.
\newblock \emph{Adv. in Appl. Probab.} \textbf{10}~(1), 85--108 (1978).
\newblock \href{http://www.ams.org/mathscinet-getitem?mr=MR0480208}{MR0480208}.

\bibitem[{Mourrat(2012)}]{MR2923190}
J.-C. Mourrat.
\newblock Lyapunov exponents, shape theorems and large deviations for the
  random walk in random potential.
\newblock \emph{ALEA, Lat. Am. J. Probab. Math. Stat.} \textbf{9}, 165--211
  (2012).
\newblock \href{http://www.ams.org/mathscinet-getitem?mr=MR2923190}{MR2923190}.

\bibitem[{Zhang(1993)}]{MR1245289}
Y.~Zhang.
\newblock A shape theorem for epidemics and forest fires with finite range
  interactions.
\newblock \emph{Ann. Probab.} \textbf{21}~(4), 1755--1781 (1993).
\newblock \href{http://www.ams.org/mathscinet-getitem?mr=MR1245289}{MR1245289}.

\end{thebibliography}
\end{document}